\documentclass[12pt]{article}
\usepackage{amssymb,amsfonts,amsmath,amsthm,comment,dsfont,xcolor,tikz,tikz-cd}
\usepackage[colorlinks, citecolor=purple,linkcolor=blue,urlcolor=teal]{hyperref}
\usepackage[lite,alphabetic]{amsrefs}
\usepackage{cleveref}
\usepackage[a4paper, portrait, margin=24mm]{geometry}

\newcommand{\1}{\mathds{1}}
\newcommand{\0}{\mathds{O}}

\newcommand{\Q}{\mathbb{Q}}
\newcommand{\R}{\mathbb{R}}

\newcommand{\N}{\mathbb{N}}

\newcommand{\Ba}{\mathrm{B}}

\newcommand{\So}{\mathrm{S}}

\newcommand{\8}{\infty}
\newcommand{\ph}{\mathrm{ph}}
\newcommand{\oc}{\mathrm{oc}}
\newcommand{\soc}{\mathrm{\sigma oc}}

\newcommand{\Sol}{\mathrm{Sol}}

\newcommand{\Co}{\mathcal{C}}

\newcommand{\bbot}{\bot\bot}

\newcommand{\too}{\xrightarrow[]{\mathrm{o}}}
\newcommand{\touo}{\xrightarrow[]{\mathrm{uo}}}
\newcommand{\toso}{\xrightarrow[]{\sigma\mathrm{o}}}

\newcommand{\touso}{\xrightarrow[]{\mathrm{u\sigma o}}}

\newcommand{\toww}{\xrightarrow[]{{\mathrm{w}^{*}}}}
\newcommand{\tobw}{\xrightarrow[]{{\mathrm{bw}^{*}}}}

\newtheoremstyle{mytheorem}
  {12pt}   
  {4pt}   
  {\itshape} 
  {}       
  {\bfseries} 
  {.}      
  {.5em}   
  {}       

\newtheoremstyle{mydefinition}
  {12pt}        
  {8pt}       
  {\normalfont} 
  {}           
  {\bfseries}   
  {.}           
  {.5em}        
  {}

\newcounter{dummy} \numberwithin{dummy}{section}
\theoremstyle{mytheorem}
\newtheorem{theorem}[dummy]{Theorem}
\newtheorem{lemma}[dummy]{Lemma}

\newtheorem{proposition}[dummy]{Proposition}
\newtheorem{corollary}[dummy]{Corollary}
\newtheorem{question}[dummy]{Question}
\theoremstyle{mydefinition}
\newtheorem{remark}[dummy]{Remark}
\newtheorem{example}[dummy]{Example}
\usepackage{scalerel}
\DeclareMathOperator*{\bigplus}{\scalerel*{+}{\sum}}

\begin{document}

\title{Variants of order semicontinuity in Banach lattices}

\author{Eugene Bilokopytov\thanks{Instituto de Ciencias Matem\'aticas, Madrid, Spain (\texttt{bilokopy@ualberta.ca}).}}
\maketitle

\begin{abstract}
In this article we consider various properties of a normed lattice, which are similar to semicontinuity (also known as the Fatou property), establish relations between these properties, and discuss stability of these properties under renorming. In particular, we show that a normed lattice $F$ is order continuous iff every renorming of $F$ is semicontinuous. We also prove that the weakly Fatou (we propose the term ``demicontinuous'') normed lattices are precisely the ones isomorphic to regular sublattices of monotonically complete Banach lattices. In order to do so we introduce the concept of the Lorentz completion of a demicontinuous normed lattice, which is somewhat analogous to the universal completion from the vector lattice theory. Furthermore, we unify and simplify the proofs of the characterizations of monotone completeness from \cite{aw} and \cite{taylor} and provide their quantitative versions.

The semicontinuity-related properties in AM-spaces have some additional features. While the classical Kakutani theorem states that AM-spaces are precisely the closed sublattices of $\Co\left(K\right)$-spaces, we show, using two different methods, that the semicontinuous AM-spaces are precisely the closed regular sublattices of $\Co\left(K\right)$-spaces. Finally, we prove that for a normed space $E$, the AM-space of positively homogeneous weak* continuous functions on $\Ba_{E^{*}}$ is semicontinuous iff it is a regular sublattice of $\Co\left(\Ba_{E^{*}},\mathrm{w}^{*}\right)$ iff $\dim E<\8$.\medskip


MSC2020 46A19, 46A40, 46B42, 46E05, 46E15.
\end{abstract}

\section{Introduction}

Semicontinuity (a.k.a. the Fatou property)\footnote{There is a lot of ambiguity in the terminology used in the literature. The term ``Fatou property'' in the Banach lattice literature means what we call ``semicontinuity'', but in Banach function space literature it is usually means what we call ``$1$-monotone completeness''; what we call ``monotone completeness'' is often referred to as ``Levi property'', but the latter term sometimes means what we call ``monotone boundedness''. For this reason, we propose to use the unambiguous term ``order semicontinuity'', which has been used in the Japanese and Soviet literature, and develop an appropriate nomenclature for various related properties.} is classical in the Banach lattice theory. Less well known are demicontinuity (this is our proposed term for the weak Fatou property) and the concept of a semicontinuous (a.k.a. Fatou) topology. This article is dedicated to clarification of the relationships between these and related concepts in the context of general normed lattices, as well as in the context of AM-spaces.\medskip

Recall that semicontinuity can be stated as order closedness of the unit ball of a normed lattice. A locally solid topology on a vector lattice is called semicontinuous if it has a local base consisting of order closed sets. Clearly, every semicontinuous normed lattice has a semicontinuous topology, but a semicontinuous normed lattice can have a non-semicontinuous renorming (which nevertheless generates the same topology, which is semicontinuous). If a normed lattice has a semicontinuous topology, it is then demicontinuous, i.e. the order limits of nets in the unit ball have a bounded norm.

We prove in Theorem \ref{oc} that a normed lattice is order continuous if and only if it has no non-semicontinuous renormings. On the other hand, we show (Proposition \ref{qf}) that a normed lattice has a semicontinuous topology if and only if it has a semicontinuous renorming, which is also equivalent to the fact that the order closure of the unit ball is norm bounded. At the same time demicontinuity is equivalent to the fact that the order adherence of the unit ball is norm bounded (Proposition \ref{wf}). We introduce pointwise demicontinuity for normed lattices, which may be viewed as the ``limit case'' of demicontinuity. It turns out (Proposition \ref{barely}) that this property is equivalent to the fact that every directed norm bounded set is dominable, i.e. order bounded in the universal completion. 

In Theorem \ref{awt} we present a simplified proof of the characterizations of monotone completeness from \cite{aw} and \cite{taylor} and provide their quantitative versions in Theorem \ref{rawt}. A more general property of monotone boundedness is characterized in Theorem \ref{lev}. It turns out (Corollary \ref{lorentz}) that a normed lattice is demicontinuous if and only if it embeds regularly into a monotonically complete Banach lattice. In order to prove this we introduce a completion of a demicontinuous normed lattice which we call the Lorentz completion. Its properties, which are somewhat analogous to those of the universal completion, are explored in Theorem \ref{delta}. Note that the notion of monotone completeness is also analogous to that of universal completeness, but in the presence of a norm. In order to deal with this and related topics we have thoroughly analyzed order adherence of solid sets in vector lattices in Section \ref{s2}.

Recall that AM-spaces are precisely the Banach lattices which embed isometrically as closed sublattices of spaces of the form $\Co\left(K\right)$, for some compact Hausdorff space $K$. One may ask when the embedding can be chosen order continuous, i.e. when the space embeds as a regular sublattice of $\Co\left(K\right)$. We prove, using two different methods, in Theorems \ref{main} and \ref{david} that semicontinuous AM-spaces are exactly the Banach lattices of this kind. On the other hand, we show in Theorem \ref{regu} that any AM-space embeds as a regular sublattice of $\Co\left(X\right)$, for some $\sigma$-compact Hausdorff space $X$.

For a normed space $E$ the space $\Co_{\ph}\left(\Ba_{E^{*}},\mathrm{w}^{*}\right)$ of positively homogeneous weak* continuous functions on $\Ba_{E^{*}}$ is a closed sublattice of $\Co\left(\Ba_{E^{*}},\mathrm{w}^{*}\right)$, hence an AM-space. This space is closely related to the free Banach lattice over $E$, and has been studied recently. We show (Theorem \ref{cph}) that $\Co_{\ph}\left(\Ba_{E^{*}},\mathrm{w}^{*}\right)$ is semicontinuous iff it is pointwise demicontinuous iff it is regular in $\Co\left(\Ba_{E^{*}},\mathrm{w}^{*}\right)$ iff $\dim E<\8$, thus providing a large supply of AM-spaces which are not pointwise demicontinuous. One can also consider the space $\Co_{\ph}\left(E^{*},\mathrm{bw}^{*}\right)$ of positively homogeneous functions on $E^{*}$, which are continuous with respect to the bounded weak* topology on $E^{*}$. It is easy to see that this space is isomorphic to $\Co_{\ph}\left(\Ba_{E^{*}},\mathrm{w}^{*}\right)$, and we show in Corollary \ref{regu2} that it is regular in $\Co\left(E^{*},\mathrm{bw}^{*}\right)$.

In Sections \ref{s9} and \ref{s10} we consider the countable analogues of the properties studied in this article. Most of the results can be transferred into the countable setting.\medskip

We conclude the introduction with a quick review of gauges. Given a convex balanced set $B$ in a vector space $E$, the \emph{Minkowski functional} or a \emph{gauge} of $B$ is defined by $\|e\|_{B}:=\bigwedge\left\{r\ge 0,~ e\in rB\right\}$. Note that $\|e\|_{B}<\8$ if and only if $e\in E_{B}:=\bigcup\limits_{n\in\N}nB$, which is a linear subspace of $E$. It is not hard to show that $\|\cdot\|_{B}$ is a seminorm on $E_{B}$ and the unit ball with respect to $\|\cdot\|_{B}$ is $\bigcap\limits_{r>1}rB$. Note that $\ker\|\cdot\|_{B}=\bigcap\limits_{n\in\N}\frac{1}{n}B=\left\{e\in E,~ \forall n\in\N:~ ne\in B\right\}$, in particular $\|\cdot\|_{B}\equiv 0$ iff $B=E$. We call $B$ \emph{radially bounded} if it contains no rays emanating from $0_{E}$, i.e. if for every $e\ne 0_{E}$ there is $n\in\N$ such that $ne\notin B$. This condition is equivalent to the fact that $\|\cdot\|_{B}$ is a norm. We also call $B$ \emph{radially closed} if for every $e\in E$ such $\left(1-\frac{1}{n}\right)e\in B$, for every $n\in\N$, we have $e\in B$. Equivalently, the unit ball with respect to $\|\cdot\|_{B}$ is equal to $B$. Finally, note that if $E$ is a vector lattice, and $B$ is solid, then $\|\cdot\|_{B}$ is a solid seminorm on $E_{B}$.

\section{Order adherence of a solid set}\label{s2}

In this section, we develop the necessary background from vector lattice theory that will be used later. Throughout this section, $F$ is an Archimedean vector lattice. We start by briefly introducing order and unbounded order convergences. For details about the results mentioned in this paragraph and further information, the reader is referred to, e.g., \cite{bctv}. If $P\subset F$ is upward directed, we will often view it as a net indexed by itself, while if $P$ is downward directed, we will view it as a net indexed by $-P$.

We say that a net $\left(f_{p}\right)_{p\in P}\subset F$ converges \emph{in order} to $f\in F$ (denoted $f_{p}\too f$) if there is $Q\subset F$ with $Q\downarrow 0_{F}$ which \emph{dominates the tails} of $\left(f_{p}-f\right)_{p\in P}$, i.e. for each $q\in Q$ there is $p_{q}$ such that $\left|f_{p}-f\right|\le q$, whenever $p\ge p_{q}$. Next, we say that $\left(f_{p}\right)_{p\in P}$ \emph{unbounded order (uo)} converges to $f$ (denoted $f_{p}\touo f$) if $\left|f_{p}-f\right|\wedge e\too 0_{F}$, for every $e\in F_{+}$. It is easy to see that for an increasing net its supremum is also its order and uo limit (in particular for such a net the supremum and both limits exist simultaneously). Order convergence implies uo convergence to the same limit; the converse is true for order bounded nets. The linear and lattice operations are continuous on $F$, with respect to both order and uo convergences. The \emph{order adherence} of $B\subset F$ is the set of all existing order limits of nets in $B$. We say that $B$ is \emph{order closed} if its order adherence equals $B$. The \emph{order closure}  of $B$ is the intersection of all order closed sets containing $B$. This set can be obtained by transfinitely iterating the order adherence. Uo adherence and closure are defined similarly.

Let $E\subset F$ be a sublattice. For $f,g\in F$ we denote $\left[f,g\right]_{E}:=\left[f,g\right]\cap E$. Recall that $E$ is \emph{order dense} in $F$ if for every $f>0_{F}$ there is $e\in E$ such that $0_{F}<e\le f$. Equivalently, for every $f\in F_{+}$ we have $f=\bigvee\left[0_{F},f\right]_{E}$ (see \cite[Theorem 1.27]{ab0}).

\begin{lemma}\label{solid}
Let $E\subset F$ be an order dense sublattice and let $B\subset F$ be solid. The following sets are equal:
\item[(i)] The set of all $f\in F$ such that there is $Q\subset B_{+}$ with $Q\uparrow\left|f\right|$;
\item[(i')] The set of all $f\in F$ such that there is $Q\subset B_{+}\cap E$ with $Q\uparrow\left|f\right|$;
\item[(ii)] The order adherence of $B$ in $F$;
\item[(ii')] The order adherence of $B\cap E$ in $F$;
\item[(iii)] The uo adherence of $B$ in $F$;
\item[(iii')] The uo adherence of $B\cap E$ in $F$.\medskip

This set is solid and its intersection with $E$ is the common order and uo adherence of $B\cap E$ in $E$. The order and uo closures of $B$ and $B\cap E$ also coincide, and are a solid set. Both the adherence and the closure are convex, whenever $B$ is.
\end{lemma}

Taking $E=F$ in the main claim of Lemma \ref{solid} recovers the following well known (but not obvious) fact.

\begin{corollary}\label{solid1}
If $B\subset F$ is solid, then both of its order and uo adherences consist of $f\in F$ such that there is $Q\subset B_{+}$ with $Q\uparrow \left|f\right|$.
\end{corollary}

\begin{proof}[Proof of Lemma \ref{solid}]
First, it is clear that the sets derived from $B\cap E$ are contained in sets derived from $B$ in a similar way. Also, since order convergence implies uo convergence to the same limit, it follows that the sets derived using order adherence are contained in the sets derived using uo adherence. It is also clear that the sets in (i) and (i') are contained in the sets in (ii) and (ii'), respectively. Hence, in order to complete the proof of the equivalences it is enough to show that the set in (iii) is contained in the set in (i').\medskip

Let $f$ be in the uo adherence of $B$, so that there is $\left(f_{p}\right)_{p\in P}\subset B$ which uo converges to $f$. It then follows that $B\ni \left|f\right|\wedge\left|f_{p}\right|\touo \left|f\right|$. For $p\in P$ let $e_{p}:=\left|f\right|\wedge\left|f_{p}\right|$; the net $\left(e_{p}\right)_{p\in P}$ is order bounded and uo-convergent to $\left|f\right|$, hence $e_{p}\too \left|f\right|$. Therefore, there is $Q\subset F$ such that $Q\downarrow 0_{F}$, which dominates the tails of $\left(\left|f\right|-e_{p}\right)_{p\in P}$.

Fix $q\in Q$. There is $p\in P$ such that $\left|f\right|-e_{p}\le q$, so that $g_{q}:=\left(\left|f\right|-q\right)^{+}\le e_{p}\in B$. As $B$ is solid, it follows that $g_{q}\in B$, and so $\left[0_{F},g_{q}\right]_{E}\subset B\cap E$.

Let us show that $S:=\bigcup\limits_{q\in Q}\left[0_{F},g_{q}\right]_{E}$ is directed. If $s,t\in S$, there are $q,r\in Q$ such that $s\in \left[0_{F},g_{q}\right]_{E}$ and $t\in \left[0_{F},g_{r}\right]_{E}$. Find $u\in Q$ such that $u\le q,r$. Then, $g_{u}\ge g_{q},g_{r}$, hence $s\vee t\in E$, and $s\vee t\le g_{u}$. Thus, $s\vee t\in \left[0_{F},g_{u}\right]_{E}\subset S$. Also, note that since $E$ is order dense in $F$, it follows that $g_{q}=\bigvee\left[0_{F},g_{q}\right]_{E}$, for every $q\in Q$. Finally, $\left(\left|f\right|-q\right)^{+}\uparrow \left|f\right|$ yields $S\uparrow \left|f\right|$. This concludes the proof of the main claim, and consequently Corollary \ref{solid1} is also proven.\medskip

Let $C$ be the obtained set. We now prove that $C$ is solid. Assume that $f\in\Sol C$, so that there is a net $\left(g_{p}\right)_{p\in P}\subset B$ which order converges to $g$ with $\left|f\right|\le \left|g\right|$. Since $B$ is solid, by replacing $g_{p}$ with $\left|g_{p}\right|$, for every $p\in P$, we may assume that $g_{p}\in B_{+}$, and also that $g\in F_{+}$. Then, $f_{p}:=g_{p}\wedge f^{+}-g_{p}\wedge f^{-}$ satisfies $\left|g_{p}\wedge f^{+} - g_{p}\wedge f^{-}\right|=g_{p}\wedge \left|f\right|\le g_{p}$, hence $f_{p}\in B$, for every $p\in P$. Moreover, $g_{p}\wedge f^{\pm}\too g\wedge f^{\pm}=f^{\pm}$ yields $f_{p}\too f$. Thus, $f\in C$.\medskip

As the order and uo closures are equal to the iterated adherences, one can easily prove by (transfinite) induction that order and uo closures of $B$ and $B\cap E$ also coincide, and are a solid set. It follows that $C\cap E$ is the set of all existing in $E$ order (or uo) limits of nets in $B\cap E$, hence the common order and uo adherence of $B\cap E$ in $E$. Since order convergence is linear, the order limit preserves convex combinations, and so if $B$ is convex, then so is $C$. The case of closure can also be accomplished by induction.
\end{proof}

\begin{lemma}\label{conve}
If $B\subset F$ is convex and closed under taking the absolute value, then $\Sol B$ is convex.
\end{lemma}
\begin{proof}
Let $e,f\in\Sol B$ and let $s,t\ge 0$ be such that $s+t=1$. Let $g,h\in B$ be such that $\left|e\right|\le\left|g\right|$ and $\left|f\right|\le\left|h\right|$. Our assumptions about $B$ yield $s\left|g\right|+t\left|h\right|\in B$, and so $\left|se+tf\right|\le s\left|e\right|+t\left|f\right|\le s\left|g\right|+t\left|h\right|$ implies $se+tf\in\Sol B$. Thus, $\Sol B$ is convex.
\end{proof}

\begin{corollary}\label{solid2}
Let $E\subset F$ be an order dense sublattice and let $B\subset E$ be solid in $E$. Then, the order and uo adherences of $B$ in $F$ and the order and uo adherences of $\Sol_{F} B$ all coincide. This set is always solid, and it is convex whenever $B$ is. $f\in F$ belongs to this set if and only if there is $Q\subset B_{+}$ with $Q\uparrow\left|f\right|$.
\end{corollary}
\begin{proof}
Take $C:=\Sol_{F} B$ and note that $C\cap E=B$. It follows from Lemma \ref{conve} that if $B$ is convex then so is $C$. The result now follows from Lemma \ref{solid}.
\end{proof}

\begin{proposition}\label{solid3}
Assume that $F$ is endowed with a Hausdorff order continuous locally solid topology $\tau$ and $B\subset F$ is solid. Then, $\overline{B}^{\tau}$ coincides with the common second order and uo adherence of $B$.
\end{proposition}
\begin{proof}
Let $C$ be the order adherence of $B$. Since $\tau$ is an order continuous topology, and $\overline{B}^{\tau}$ is $\tau$-closed, it is order closed, and hence it contains $C$, as well as its order adherence, and so we only need to prove the converse inclusion. Moreover, since all the sets involved are solid, it is enough to show that if $f\in \overline{B}^{\tau}_{+}$, there is $P\subset C$ such that $P\uparrow f$.

Let $H$ be the carrier of $\tau$ (see \cite[Definition 4.15 or Exercise 4.16(c)]{ab0}), which is order dense in $F$ (see \cite[Theorem 4.17(b)]{ab0}). According to \cite[Theorems 4.17(c) and 4.19]{ab0}, every $\tau$-null net in $H$ contains an order-null sequence. Assume first that $f\in H_{+}$. There is a net $\left(f_{p}\right)_{p\in P}\subset B$ which $\tau$-converges to $f$. Then, $B\cap H\ni f\wedge f_{p}^{+}\to f$ in $\tau$, and so there is a sequence contained in this net which order converges to $f$. Hence, $f\in C$. In the general case, we have that $P:=\left[0_{F},f\right]_{H}\uparrow f$ due to order density of $H$, and moreover, since $\overline{B}^{\tau}$ is solid, we have $P\subset \overline{B}^{\tau}\cap H\subset C$, as required.
\end{proof}

The following example shows that in the context of Proposition \ref{solid3} the first order adherence of $B$ may be strictly contained in $\overline{B}^{\tau}$.

\begin{example}\label{secon}
Let $I$ be the unit interval $\left[0,1\right]$, let $U_{n}:=\left(-\frac{1}{n},\frac{1}{n}\right)$, and let $\lambda$ be the Lebesgue measure on $I$. For a finite $K\subset I$ and $n\in\N$ let $K^{n}:=\left(K+U_{mn}\right)\cap I$, where $m$ is the cardinality of $K$. We have that $K^{n}$ is open in $I$, contains $K$ and $\lambda\left(K^{n}\right)\le \frac{2}{n}$. 

Let $F:=\R^{I}\oplus L_{0}\left(I,\lambda\right)$, where the first summand is endowed with the pointwise topology, while the second -- with the topology of convergence in measure. For a finite $K\subset I$ and $n\in\N$ let $e_{K,n}:=\1_{K}\oplus \1_{I\backslash K^{n}}$ and  $B_{K,n}:=\left\{g\in F,~ \left|g\right|\le e_{K,n}\right\}$. Let $B:=\bigcup\limits_{K,n}B_{K,n}$, where $n\in\N$ and $K$ runs over finite subsets of $I$. Clearly, $B$ is solid. It is easy to see that $f:=\1\oplus\1$ is contained in the closure of $\left\{e_{K,n}\right\}_{K,n}$, hence $f\in \overline{B}$.

Assume that $f$ is in the order adherence of $B$. Then, according to Corollary \ref{solid1} there is a directed $P\subset B_{+}$ such that $P\uparrow f$. Then, $\left(f-P\right)\downarrow 0_{F}$ and for every $p\in P$ there is $K_{p}, n_{p}$ such that $p\le e_{K_{p},n_{p}}$, hence $0_{F}\le f-e_{K_{p},n_{p}}\le f-p$. It follows that $e_{K_{p},n_{p}}\too 0_{F}$, and projecting onto the second coordinate yields $\1_{K_{p}^{n_{p}}}\too \0$. Since $L_{0}\left(I,\lambda\right)$ is order complete it follows that $\bigwedge\limits_{p\in P}\bigvee\limits_{q\ge p}\1_{K_{p}^{n_{p}}}=\0$.

For $p\in P$, consider the decomposition $p=p^{1}\oplus p^{2}$, where $p^{1}\in \R^{I}$ and $p^{2}\in L_{0}\left(I,\lambda\right)$. We have that $p^{1}\uparrow_{p\in P}\1$, and so $p^{1}\left(t\right)\uparrow_{p\in P}1$, for every $t\in I$.

Fix $p\in P$ and for every $t\in I$, find $q_{t}\ge p$ such that $p^{1}\left(t\right)>0$. Since $p^{1}\le \1_{K_{p}}$, it follows that $t\in K_{p_{t}}$, and so $K_{p_{t}}^{n_{p_{t}}}$ is an open neighborhood of $t$. By compactness there are $t_{1},...,t_{m}\in I$ such that $I=\bigcup\limits_{k=1}^{m}K_{p_{t_{k}}}^{n_{p_{t_{k}}}}$. Then, $\bigvee\limits_{k=1}^{m}\1_{K_{p_{t_{k}}}^{n_{p_{t_{k}}}}}=\1$, hence $\bigvee\limits_{q\ge p}\1_{K_{q}^{n_{q}}}=\1$. As $p$ was arbitrary we get $\bigwedge\limits_{p\in P}\bigvee\limits_{q\ge p}\1_{K_{p}^{n_{p}}}=\1$, contradicting the previous conclusion.
\qed\end{example}

Recall that $F$ has the \emph{principal projection property (PPP)}, if every band of the form $\left\{e\right\}^{\bbot}$, for $e\in F$ is a projection band. In the next result we use the fact that if $P$ is the projection onto $\left\{\left(f-g\right)^{+}\right\}^{\bbot}$, then $Pf\ge Pg$. Indeed, since $\left(f-g\right)^{+}\bot \left(f-g\right)^{-}$, it follows that $P\left(f-g\right)^{-}=0_{F}$, hence $Pf-Pg=P\left(f-g\right)=P\left(f-g\right)^{+}-P\left(f-g\right)^{-}=P\left(f-g\right)^{+}\ge 0_{F}$.

Also, recall that $f$ is a \emph{fragment} of $e$ (also known as \emph{component}) if $f\bot e-f$.

\begin{proposition}\label{eps}
Assume that $F$ is a vector lattice with PPP and let $Q\subset F_{+}$ and $e\in F$ be such that $e=\bigvee Q$. Then, for every $0\le r<1$ there is a disjoint $D\subset F_{+}$ such that $\bigvee D=re$ and for every $d\in D$ there is $q\in Q$ such that $d\le q$. Moreover, if $Q$ is countable then $D$ can also be chosen to be countable.
\end{proposition}
\begin{proof}
By Zorn's lemma we can find a maximal (with respect to inclusion) disjoint $D\subset F_{+}$ such that every $d\in D$ is a fragment of $re$ and there is $q\in Q$ such that $d\le q$. Assume that $re$ is not the supremum of $D$, so that there is $f>0_{F}$ such that $re-d\ge f$, for every $d\in D$. Clearly, $re\ge f$. Note that $re-d\bot d$, hence $f\bot d$, for every $d\in D$. Since $$\bigvee\limits_{q\in Q}\left(f\wedge \left(q-re\right)^{+}\right)=f\wedge\left(\bigvee Q-re\right)^{+}=f\wedge \left(1-r\right)e\ge f\wedge \frac{1-r}{r}f >0_{F},$$ there is $q\in Q$ such that $g:=f\wedge \left(q-re\right)^{+}>0_{F}$. Let $h:=P_{g}re$, which is a fragment of $re$. Moreover, $h\in \left\{g\right\}^{\bbot}\subset \left\{f\right\}^{\bbot}$ guarantees that $h\bot D$. We have $h\le P_{\left(q-re\right)^{+}}re\le P_{\left(q-re\right)^{+}}q\le q$. Hence, $D\cup\left\{h\right\}$ exhibits non-maximality of $D$. This contradiction shows that $re=\bigvee D$.\medskip

We now consider the case when $Q$ is countable, and so can be denumerated as $\left\{q_{n}\right\}_{n\in\N}$. Let $d_{1}:=P_{\left(q_{1}-re\right)^{+}}re$, and if $d_{1},...,d_{n}$ are selected, put $d_{n+1}:=P_{\left(q_{n+1}-re\right)^{+}}\left(re-d_{1}-...-d_{n}\right)$, which is a fragment of $re$ (induction). Clearly, $\left\{d_{n}\right\}_{n\in\N}$ is a disjoint sequence, and moreover, $d_{n}\le P_{\left(q_{n}-re\right)^{+}}re\le P_{\left(q_{n}-re\right)^{+}}q_{n}\le q_{n}$, for every $n\in\N$. For every $n\in\N$ we have $0_{F}\le re-d_{1}-...-d_{n}\bot \left(q_{1}-re\right)^{+},...,\left(q_{n}-re\right)^{+}$. Let $d\le re$ be an upper bound for $\left\{d_{n}\right\}_{n\in\N}$. Then, for every $n\in\N$ we have $0_{F}\le re-d\le re-d_{1}-...-d_{n}$, hence $\left(re-d\right)\wedge\left(q_{n}-re\right)^{+}=0_{F}$. Therefore, $$0_{F}=\bigvee\limits_{n\in\N}\left(\left(re-d\right)\wedge\left(q_{n}-re\right)^{+}\right)=\left(re-d\right)\wedge\left(\bigvee\limits_{n\in\N}q_{n}-re\right)^{+}=\left(re-d\right)\wedge \left(1-r\right)e,$$
which implies that $re-d=0_{F}$. It follows that $d=re$, and so $\bigvee D=re$.
\end{proof}

\begin{remark}\label{smp}
Note that for the ``uncountable case'' the ``PPP'' assumption can be relaxed to ``SMP'' (\emph{sufficiently many projections}, i.e. the property that every nonzero band in $F$ contains a nonzero projection band).

It is not always true that in the context of Proposition \ref{eps} there is a disjoint $D\subset F_{+}$ such that $\bigvee D=e$ and for every $d\in D$ there is $q\in Q$ such that $d\le q$, even if $Q$ is an increasing sequence, and under various additional assumptions about $F$. Let $F:=L_{p}\left[0,1\right]$, for $1\le p\le \8$, enumerate $\Q\cap\left(0,1\right)=\left\{q_{n}\right\}_{n\in\N}$, and for every $n\in\N$ let $f_{n}$ be a function such that $f_{n}\left(0\right)=0=f_{n}\left(1\right)$, $f_{n}\left(q_{n}\right)=1$, and $f_{n}$ is affine on $\left[0,q_{n}\right]$ and $\left[q_{n},1\right]$. It is easy to see that the pointwise supremum of $\left\{f_{n}\right\}_{n\in\N}$ is $\1_{\left(0,1\right)}$, and so $g_{n}\uparrow\1$, in $F$, where $g_{n}=f_{1}\vee...\vee f_{n}$, $n\in\N$. Assume that $D\subset F_{+}$ is disjoint and such that for every $d\in D$ there is $n\in \N$ with $d\le g_{n}$. We claim that $\1$ is not the supremum of $D$. Indeed, fix $d\in D$ and find $n\in\N$ such that $d\le g_{n}$. There are $0<r\le s<1$ and a set $U\subset\left[0,1\right]$ of positive measure such that $r\1_{U}\le d\wedge \1_{U}\le s\1_{U}$. Then, every $d\ne e\in D$ vanishes on $U$, and so $e\le \1_{\left[0,1\right]\backslash U}$. We conclude that $D\le \1_{\left[0,1\right]\backslash U}+s\1_{U}<\1$.
\qed\end{remark}

\section{Semicontinuity and demicontinuity}\label{s3}

Throughout the section $F$ is a normed lattice. Let us recall some important classes of normed lattices. Namely, we will say that $F$ is

\begin{itemize}
\item an \emph{AM-space} if it is norm complete, and satisfies $\|e\vee f\|=\|e\|\vee\|f\|$, for every $e,f\in F_{+}$;
\item \emph{order continuous} if whenever $P\subset F$ is such that $P\downarrow 0_{F}$, it then follows that $P\to 0_{F}$ in norm, that is $\bigwedge\limits_{p\in P}\|p\|=0$;
\item \emph{(order)\footnote{For better readability, throughout the article we will be dropping the word ``order'' from this term and similar terms. This does not seem to create ambiguity.} semicontinuous\footnote{Also known as Fatou.}} if whenever $P\subset\Ba_{F_{+}}$ and $f\in F$ are such that $P\uparrow f$ it then follows that $\|f\|\le 1$, i.e. $f\in \Ba_{F}$.
\end{itemize}

Note that $F$ is order continuous if and only if order convergence of a net in $F$ implies norm convergence to the same limit (see e.g. \cite[Theorem 3.11]{bctv}). It is well-known that the only order continuous AM-spaces are normed lattices of the form $c_{0}\left(X\right)$, for some set $X$ (see e.g. \cite[Theorem 5.7]{bgdmt}). Also, note that order continuity is a topological property and so it is preserved under renorming.

Let us consider an important example of a semicontinuous AM-space.

\begin{example}\label{ous}
Recall that $e\in F_{+}$ is a \emph{strong unit} if for every $f\in F$ there is $r\ge 0$ such that $\left|f\right|\le re$. Then, $\|\cdot\|_{e}$ defined by $\|f\|_{e}:=\bigwedge\left\{r\ge 0,~ \left|f\right|\le re\right\}$ is the gauge of $\left[-e,e\right]$. The latter interval is radially closed and radially bounded, and so $\|\cdot\|_{e}$ is a norm, and $\left[-e,e\right]$ is the unit ball with respect to $\|\cdot\|_{e}$. It is now easy to see that $\|\cdot\|_{e}$ is a semicontinuous norm. We also recall Krein--Kakutani theorem (see \cite[Theorem 45.3]{lz}) which asserts that there is an injective homomorphism $J$ from $F$ onto a dense sublattice of $\Co\left(K\right)$, for some compact Hausdorff space, such that $Je=\1$, and so $J$ is an isometry from $\left(F,\|\cdot\|_{e}\right)$ into $\Co\left(K\right)$ with the supremum norm. Hence, $\left(F,\|\cdot\|_{e}\right)$ is an AM-space iff it is $\|\cdot\|_{e}$-complete iff it is isometrically isomorphic to $\Co\left(K\right)$.
\qed\end{example}

A semicontinuous AM-space can have a non-semicontinuous AM-renorming.

\begin{example}\label{cn}
Let $c$ be the space of convergent sequences, which is semicontinuous, by Example \ref{ous}. For $n\in\N$ let $c_{n}:=\left(c,\|\cdot\|^{n}\right)$, where $\|\cdot\|^{n}$ is a norm on $c$ defined by $\left\|\left(t_{k}\right)_{k\in\N}\right\|^{n}:=\bigvee\limits_{k\in\N}\left|t_{k}\right|\vee n\left|\lim\limits_{k\in\N}t_{k}\right|$. Note that $\|\cdot\|^{n}$ is an AM-norm, equivalent to the supremum norm on $c$. Let us show that if $n\ne 1$ then $c_{n}$ fails to be semicontinuous. Let $\left(e_{k}\right)_{k\in\N}$ be the standard basis of $c_{0}$, let $f_{k}:=e_{1}+...+e_{k}$, for $k\in\N$, and let $\1$ be the constant one sequence. It is clear that $f_{k}\uparrow \1$, but $f_{k}\in\Ba_{c_{n}}$, for every $k\in\N$, while $\left\|\1\right\|^{n}=n$.
\qed\end{example}

It turns out that many semicontinuous normed lattices admit non-semicontinuous renormings, as explained by the next result.

\begin{theorem}\label{oc}
$F$ is order continuous if and only if every renorming of $F$ is semicontinuous.
\end{theorem}
\begin{proof}
Necessity: It is enough to show that every order continuous normed lattice is semicontinuous. Let $P\subset\Ba_{F_{+}}$ and $f\in F$ be such that $P\uparrow f$. Then, $f-P\downarrow 0_{F}$, hence $f-P\to 0_{F}$, and so $P\to f$. Since $\Ba_{F}$ is norm closed, it follows that $f\in\Ba_{F}$.\medskip

Sufficiency: Assume that $F$ is not order continuous, so that there are a directed $P\subset F_{+}$ and $f\in F$ such that $P\uparrow f$, but $\|f-p\|\ge 2$, for every $p\in P$. Note that $\bigcup\limits_{p\in P}\left[-p,p\right]$ is solid, convex (as a directed union of convex sets) and order bounded (because it is contained in $\left[-f,f\right]$). Let $B:=\Ba_{F}+\bigcup\limits_{p\in P}\left[-p,p\right]$ which is a norm bounded convex solid set. It is easy to see that $\overline{B}$ is a norm bounded (hence, radially bounded) norm closed (hence, radially closed) convex solid set.

If $g\in \Ba_{F}$ and $h\in \left[-p,p\right]$, for some $p\in P$, then $f-h\ge f-p\ge 0_{F}$, hence $\|f-g-h\|\ge \|f-h\|-\|g\|\ge \|f-p\|-\|g\|\ge 1$. Thus, $f\notin\overline{B}$. The gauge $\|\cdot\|_{\overline{B}}$ of $\overline{B}$ is a solid norm on $F$ equivalent to the given norm. However, $P\subset \overline{B}$ and $P\uparrow f\notin \overline{B}$ implies that $\|f\|_{\overline{B}}>1$ but $\|p\|_{\overline{B}}\le 1$, for every $p\in P$. Thus, $\|\cdot\|_{\overline{B}}$ is not a semicontinuous norm.
\end{proof}

\begin{proposition}[Cf. Theorem 4.6, \cite{ab0}]\label{fatou}
The following conditions are equivalent:
\item[(i)] $F$ is semicontinuous;
\item[(ii)] $\Ba_{F}$ is order closed;
\item[(ii')] $\Ba_{F}$ is uo closed;
\item[(iii)] If $f_{p}\too f$, then $\|f\|\le\liminf\|f_{p}\|$;
\item[(iii')] If $f_{p}\touo f$, then $\|f\|\le\liminf\|f_{p}\|$.
\end{proposition}
\begin{proof}
(i)$\Rightarrow$(ii) follows immediately from Corollary \ref{solid1}.

(ii)$\Rightarrow$(iii): Let $r:=\liminf\|f_{p}\|$. If $r=\8$, there is nothing to prove. Otherwise, let $s>r$ and let $\left(f_{p_{q}}\right)_{q\in Q}$ be a subnet of $\left(f_{p}\right)_{p\in P}$ such that $\|f_{p_{q}}\|\to t<s$. Then, there is $q_{0}$ such that $\|f_{p_{q}}\|\le s$, for every $q\ge q_{0}$. Hence, $s\Ba_{F}\ni f_{p_{q}}\too f$ implies $f\in s\Ba_{F}$, and so $\|f\|\le s$. Since $s$ was arbitrary we conclude that $\|f\|\le r$.

(iii)$\Rightarrow$(i) is trivial. (i)$\Rightarrow$(ii')$\Rightarrow$(iii')$\Rightarrow$(i) is proven similarly.
\end{proof}

Recall that a linear topology on a vector lattice is called a \emph{(order) semicontinuous topology\footnote{Also known as Fatou topology.}} if it has a local base at the origin which consists of solid order closed sets. Clearly, a semicontinuous norm induces a semicontinuous topology. The converse is false, since a renorming of a semicontinuous normed lattice may fail to be semicontinuous (see Example \ref{cn}). We say that $F$ is \emph{topologically (order) semicontinuous}, if it has an equivalent semicontinuous norm. More specifically, we will call $F$ \emph{(order) $r$-semicontinuous}, for some $r\ge 1$, if there is a semicontinuous norm $\|\cdot\|'$ on $F$ such that $\|\cdot\|\le\|\cdot\|'\le r\|\cdot\|$ (equivalently, there is a semicontinuous norm $\|\cdot\|''$ on $F$ such that $\frac{1}{r}\|\cdot\|\le\|\cdot\|''\le \|\cdot\|$). Clearly, semicontinuity implies the topological semicontinuity, and the latter property is stable under renorming. In addition to the following characterizations see also \cite[Theorem 3]{fk}.

\begin{proposition}\label{qf}
$F$ is $r$-semicontinuous if and only if the common order and uo closure of $\Ba_{F}$ is contained in $r\Ba_{F}$. Moreover, the following conditions are equivalent:
\item[(i)] $F$ is topologically semicontinuous;
\item[(ii)] The norm topology is semicontinuous;
\item[(iii)] The common order and uo closure of $\Ba_{F}$ is norm bounded;
\item[(iv)] Order closure preserves norm bounded sets;
\item[(iv')] Uo closure preserves norm bounded sets.
\end{proposition}
\begin{proof}
We start with the main claim. Necessity: Let $\|\cdot\|'$ be a semicontinuous norm on $F$ such that $\|\cdot\|\le\|\cdot\|'\le r\|\cdot\|$. Then, the unit ball $B$ with respect to $\|\cdot\|'$ satisfies $\Ba_{F}\supset B\supset \frac{1}{r}\Ba_{F}$, hence $\Ba_{F}\subset rB\subset r\Ba_{F}$. Since $B$ is order closed, it follows that the order closure of $\Ba_{F}$ is contained in $rB$, thus in $r\Ba_{F}$.\medskip

Sufficiency: Let $D\subset r\Ba_{F}$ be the order closure of $\Ba_{F}$, which is a convex solid set. It is clear that $D$ is radially bounded, but it is also radially closed, because it is order closed. The gauge $\|\cdot\|_{D}$ of $D$ is a solid seminorm on $F$ whose unit ball is $D$. As the latter is order closed, Proposition \ref{fatou} guarantees that $\|\cdot\|_{D}$ is a semicontinuous norm. At the same time $r\Ba_{F}\supset D\supset \Ba_{F}$ implies $\frac{1}{r}\|\cdot\|\le\|\cdot\|_{D}\le \|\cdot\|$.\medskip

We now prove equivalences.

(i)$\Rightarrow$(ii): Let $\|\cdot\|'$ be an equivalent semicontinuous norm. Then the balls with respect to $\|\cdot\|'$ form a local base for the norm topology of $F$ made up of solid order closed sets.

(ii)$\Rightarrow$(iii): Let $U$ be an order closed neighborhood of $0_{F}$ contained in $\Ba_{F}$. Since $\Ba_{F}$ is bounded, there is $r>0$ such that $\Ba_{F}\subset rU$. As $U$ is order closed, it follows that the order closure of $\Ba_{F}$ is contained in $rU\subset r\Ba_{F}$, and so it is norm bounded.

(iii)$\Rightarrow$(iv) and (iii)$\Rightarrow$(iv') are straightforward. (iv)$\Rightarrow$(i) and (iv')$\Rightarrow$(i) follow from the main claim.
\end{proof}

\begin{example}\label{dual}
Let $F^{\sim}_{\oc}$ be the set of all order continuous functionals on $F$, and assume that $F^{*}_{\oc}:=F^{\sim}_{\oc}\cap F^{*}$ is $r$-norming. Then, $F$ is $r$-semicontinuous. Define $\|\cdot\|'$ by $\|f\|':=\bigvee\limits_{\nu\in F^{\sim}_{\oc}\cap \Ba_{F^{*}}}\left|\nu\left(f\right)\right|$. This norm satisfies $\frac{1}{r}\|\cdot\|\le\|\cdot\|'\le\|\cdot\|$, and so it is left to show that this is a semicontinuous norm, i.e. that the set $\left\{f\in F,~ \forall \nu\in F^{\sim}_{\oc}\cap \Ba_{F^{*}}:~ \left|\nu\left(f\right)\right|\le 1\right\}$ is solid and order closed. Order closedness follows from order closedness of $\left\{f\in F,~ \left|\nu\left(f\right)\right|\le 1\right\}$, for every $\nu\in F^{\sim}_{\oc}\cap \Ba_{F^{*}}$. Solidness can be deduced from solidness of $F^{\sim}_{\oc}\cap \Ba_{F^{*}}$ and Riesz-Kantorovich theorem.
\qed\end{example}

We say that $F$ is \emph{(order) $r$-demicontinuous}, for some $r\ge 1$, if whenever $P\subset\Ba_{F_{+}}$ and $f\in F$ are such that $P\uparrow f$ it then follows that $\|f\|\le r$. Observe that $1$-demicontinuity is precisely semicontinuity. We will call $F$ \emph{(order) demicontinuous\footnote{Also known as weakly Fatou.}} if it is $r$-demicontinuous, for some $r\ge 1$. It is easy to see that demicontinuity is stable under renorming. If $F$ is $r$-semicontinuous, then it is $r$-demicontinuous. Indeed, let $\|\cdot\|'$ be a semicontinuous norm on $F$ such that $\|\cdot\|\le\|\cdot\|'\le r\|\cdot\|$; then if $P\subset\Ba_{F_{+}}$ and $f\in F$ are such that $P\uparrow f$ it then follows that $\|f\|\le\|f\|'\le \bigvee\limits_{g\in P}\|g\|'\le \bigvee\limits_{g\in P}r\|g\|\le r$. Consequently, every topologically semicontinuous normed lattice is demicontinuous. The following two results are straightforward.

\begin{proposition}\label{rwf}
The following conditions are equivalent:
\item[(i)] $F$ is $r$-demicontinuous;
\item[(ii)] The common order and uo adherence of $\Ba_{F}$ is contained in $r\Ba_{F}$;
\item[(iii)] Whenever $f_{p}\too f$, then $\|f\|\le r\liminf\|f_{p}\|$;
\item[(iii')] Whenever $f_{p}\touo f$, then $\|f\|\le r\liminf\|f_{p}\|$.
\end{proposition}

\begin{proposition}\label{wf}
The following conditions are equivalent:
\item[(i)] $F$ is demicontinuous;
\item[(ii)] The common order and uo adherence of $\Ba_{F}$ is norm bounded;
\item[(iii)] Order adherence preserves norm bounded sets;
\item[(iii')] Uo adherence preserves norm bounded sets.
\end{proposition}

Recall that a sublattice $E\subset F$ is \emph{regular} if whenever $P\subset E$ is such that $\bigwedge_{E}P=0_{F}$, it then follows that $\bigwedge_{F}P=0_{F}$. It is not hard to show that $E$ is regular if and only if the embedding of $E$ into $F$ is order continuous. Every ideal is regular, along with every order dense sublattice.

\begin{proposition}\label{pass}
$r$-demicontinuity and $r$-semicontinuity are inherited by regular sublattices, for every $r\ge 1$. In particular, semicontinuity, topological semicontinuity and demicontinuity pass down to regular sublattices. Order continuity also passes down to regular sublattices.
\end{proposition}

The following result may be viewed as a converse to Example \ref{dual}. It improves \cite[Lemma 2.4.20]{mn}.

\begin{proposition}\label{dua}
Assume that $F$ is $r$-demicontinuous and that $F^{*}_{\oc}$ separates points of $F$. Then, $F^{*}_{\oc}$ is $r^{2}$-norming.
\end{proposition}
\begin{proof}
It is enough to show that the closure $B$ of $\Ba_{F}$ in $\sigma\left(F,F^{*}_{\oc}\right)$ is contained in $r^{2}\Ba_{F}$ (see \cite[Exercise 3.90]{fhhmz}). Let $\tau$ be the absolute weak topology on $F$ generated by $F^{*}_{\oc}$, i.e. $f_{p}\to 0_{F}$ in $\tau$ if $\nu\left(\left|f_{p}\right|\right)\to 0$, for every $\nu\in F^{*}_{\oc}$. Clearly, $\tau$ is an order continuous locally solid topology, and it is Hausdorff, because $F^{*}_{\oc}$ separates points of $F$. Moreover, the dual of $\left(F,\tau\right)$ is $F^{*}_{\oc}$ (see \cite[Theorem 2.33]{ab0}), and since $\Ba_{F}$ is convex, it follows that $B=\overline{\Ba_{F}}^{\tau}$. According to Proposition \ref{solid3} $B$ is the second order adherence of $\Ba_{F}$. Hence, it follows from Proposition \ref{rwf} that $C\subset r\Ba_{F}$, hence $B\subset r^{2}\Ba_{F}$.
\end{proof}

Proposition \ref{dua} in tandem with Example \ref{dual} yield the following.

\begin{corollary}\label{wfqf}
If $F$ is demicontinuous and $F^{*}_{\oc}$ separates points of $F$, then $F$ is topologically semicontinuous.
\end{corollary}

\section{Pointwise demicontinuity and anti-demicontinuity}\label{s4}

Throughout the section $F$ is a normed lattice. $F$ is \emph{pointwise demicontinuous} if for every $f>0$ there is $t>0$ such that whenever $P\subset s\Ba_{F_{+}}$ is such that $P\uparrow f$, it then follows that $s\ge t$. If $F$ is $r$-demicontinuous, then one can take $t:=\frac{\|f\|}{r}$. Hence, every demicontinuous normed lattice is pointwise demicontinuous. It is not hard to verify that pointwise demicontinuity is stable under renorming. Note that $F$ fails to be pointwise demicontinuous if and only if there is $f>0_{F}$ such that for every $s>0$ there is $P\subset s\Ba_{F_{+}}$ with $P\uparrow f$. Alternatively, there is $f>0_{F}$ such that for every $n\in\N$ there is $P_{n}\subset \Ba_{F_{+}}$ with $P_{n}\uparrow nf$.

\begin{proposition}\label{bp}
The class of pointwise demicontinuous normed lattices is closed under $\ell_{\8}$ sums.
\end{proposition}
\begin{proof}
Let $\left\{F^{i}\right\}_{i\in I}$ be a collection of pointwise demicontinuous normed lattices and let $F:=\bigoplus_{\8}\left\{F^{i},~ i\in I\right\}$. Let $0_{F}\ne f=\left(f_{i}\right)_{i\in I}\in F_{+}$. There is $i\in I$ such that $f_{i}>0_{F^{i}}$. Let $t>0$ be such that whenever $Q\subset s\Ba_{F^{i}_{+}}$ is such that $Q\uparrow f_{i}$, it then follows that $s\ge t$. Let $T_{i}:F\to F^{i}$ be the coordinate projection, which is order continuous. If $P\subset s\Ba_{F_{+}}$ is such that $P\uparrow f$, it then follows that $T_{i}P\subset s\Ba_{F^{i}_{+}}$ and $T_{i}P\uparrow f_{i}$, which implies that $s\ge t$. We conclude that $F$ is pointwise demicontinuous.
\end{proof}

Note that in a similar way one can prove that the class of $r$-semicontinuous and $r$-demicontinuous normed lattices is also closed under $\ell_{\8}$ sums, for every fixed $r\ge 1$, including the class of semicontinuous normed lattices. However, this is not true for the classes of topologically semicontinuous and demicontinuous normed lattices, as the next example shows (it also exhibits a pointwise demicontinuous AM-space which fails to be demicontinuous).

\begin{example}[Example 7, \cite{wickstead}]\label{pcn}
For every $n\in\N$ let $c_{n}$ be as in Example \ref{cn}. Note that $\|\cdot\|^{n}$ is an AM-norm, equivalent to the supremum norm on $c$, hence topologically semicontinuous, thus pointwise demicontinuous. Let $F:=\bigoplus_{\8}\left\{c_{n},~ n\in \N\right\}$, which is pointwise demicontinuous, according to Proposition \ref{bp}. Let us show that $F$ is not demicontinuous. For every $n\in\N$ let $J_{n}:c\to F$ be the embedding into the $n$-th coordinate. Note that $\left.J_{n}\right|_{c_{0}}$ is an isometry. Let $\left(f_{k}\right)_{k\in\N}$ be as in Example \ref{cn}. It is clear that $f_{k}\uparrow \1$, hence $J_{n}f_{k}\uparrow J_{n}\1$, for every $n\in\N$. It is left to observe that $J_{n}f_{k}\in\Ba_{F}$, for every $k,n\in\N$, while $\left\|J_{n}\1\right\|=n$.
\qed\end{example}

\begin{proposition}\label{barely2}
Let $E$ be pointwise demicontinuous and let $J:F\to E$ be a continuous injective order continuous homomorphism. Then, $F$ is pointwise demicontinuous.
\end{proposition}
\begin{proof}
We may assume that $\|J\|\le 1$. Let $f>0_{F}$. Then, $Jf>0_{E}$, and so there is $t>0$ such that whenever $Q\subset s\Ba_{E_{+}}$ is such that $Q\uparrow Jf$, it then follows that $s\ge t$. Assume that $P\subset s\Ba_{F_{+}}$ is such that $P\uparrow f$. Then, $JP\subset s\Ba_{E_{+}}$, and since $J$ is order continuous, it follows that $JP\uparrow Jf$. By assumption, $s\ge t$.
\end{proof}

\begin{proposition}\label{barely3}
Assume that $E\subset F$ is an order dense sublattice which is pointwise demicontinuous in the induced norm. Then, $F$ is pointwise demicontinuous.
\end{proposition}
\begin{proof}
Let $f>0_{F}$, and find $e\in E$ with $0_{F}<e\le f$. There is $t>0$ such that whenever $Q\subset s\Ba_{E_{+}}$ is such that $Q\uparrow e$, it then follows that $s\ge t$. Assume that $P\subset s\Ba_{F_{+}}$ is such that $P\uparrow f$. Then, $Q:=\bigcup\limits_{p\in P}\left[0_{F},p\right]_{E}$ is a directed subset of $s\Ba_{E_{+}}$ such that $Q\uparrow f$ (see the argument in the proof of Lemma \ref{solid}). Therefore, $Q\wedge e$ is a directed subset of $s\Ba_{E_{+}}$ such that $Q\wedge e\uparrow f\wedge e=e$. It then follows that $s\ge t$.
\end{proof}

We will show in Example \ref{completion} that the class of pointwise demicontinuous normed lattices is not stable under completion. However, combining Proposition \ref{barely3} with \cite[Definition 5.26(i) and Theorem 5.29]{ab0} we obtain the following result.

\begin{corollary}
Assume that $F$ is pointwise demicontinuous and such that every Cauchy sequence which decreases to $0_{F}$ converges to $0_{F}$. Then, the norm completion of $F$ is pointwise demicontinuous.
\end{corollary}

It is possible to embed $F$ as an order dense sublattice of a \emph{universally complete} vector lattice (order complete, and such that every disjoint subset is order bounded, see e.g. \cite[Theorem 7.21]{ab0}). This vector lattice is essentially unique and is denoted $F^{u}$. A set in $F$ is called \emph{dominable} if it is order bounded in $F^{u}$. Note that $P\subset F_{+}$ is dominable if and only if for every $h>0_{F}$ there is $f>0_{F}$ and $k\in\N$ such that $\left(kh-g\right)^{+}\ge f$, for every $g\in P$; in other words $\bigwedge\left(kh-P\right)^{+}\ne 0_{F}$ (in the sense that the infimum either does not exist or is non-zero; see \cite[Section 7.1]{ab0}), or equivalently $\bigvee\left(P\wedge kh\right)^{+}\ne kh$. In the following result (iii)$\Rightarrow$(iv) is proven similarly to \cite[Theorem 7.50(i)]{ab0} and \cite[Proposition 2.3]{taylor}.

\begin{proposition}\label{barely}
The following conditions are equivalent:
\item[(i)] $F$ is pointwise demicontinuous;
\item[(ii)] The common order and uo adherence of $\Ba_{F}$ in $F^{u}$ is radially bounded;
\item[(iii)] The common order and uo adherence of $\Ba_{F}$ in $F$ is radially bounded;
\item[(iv)] Every directed subset of $\Ba_{F_{+}}$ is dominable.
\end{proposition}
\begin{proof}
(i)$\Rightarrow$(ii): Let $B$ be the common order and uo adherence of $\Ba_{F}$ in $F^{u}$. Let $0_{F^{u}}\ne g\in F^{u}$ and let $0_{F}<f\in F$ be such that $f\le\left|g\right|$. Let $t>0$ be such that whenever $P\subset s\Ba_{F_{+}}$ satisfies $P\uparrow f$, it then follows that $s\ge t$. Assume that $rg\in B$. Since $B$ is solid by virtue of Corollary \ref{solid2}, it follows that $rf\in B$ and so $f\in\frac{1}{r}B$. By Corollary \ref{solid2} there is $P\subset \frac{1}{r}\Ba_{F_{+}}$ with $P\uparrow f$, and so by assumption we get $\frac{1}{r}\ge t$, hence $r\le \frac{1}{t}$.\medskip

(ii)$\Rightarrow$(iii) is trivial.

(iii)$\Rightarrow$(iv): Let $P\subset F_{+}$ be an upward directed set contained in $\Ba_{F}$. Let $h>0_{F}$ and find $n\in\N$ such that $nh$ does not belong to the order adherence of $\Ba_{F}$. Then, $P\wedge nh\not\uparrow nh$, and since $h$ was chosen arbitrarily, we conclude that $P$ is dominable.\medskip

(iv)$\Rightarrow$(i): Assume that $f>0_{F}$ is such that for every $n\in\N$ there is $P_{n}\subset\frac{1}{n2^{n}}\Ba_{F_{+}}$ such that $P_{n}\uparrow f$. WLOG we may assume that $0_{F}\in P_{n}$, for every $n\in\N$. Let $P:=\bigplus\limits_{n\in\N}nP_{n}\subset \Ba_{F_{+}}$ (the set of finite sums), which is upwards directed. For every $n\in\N$ we have that $P\supset nP_{n}$, and since $nP_{n}\uparrow nf$, it follows that $\bigwedge\left(nf-nP_{n}\right)^{+}=0_{F}$, and so $\bigwedge\left(nf-P\right)^{+}=0_{F}$. We conclude that $P$ is not dominable.
\end{proof}

Recall that if $F$ is a normed lattice, the \emph{Lorentz seminorm} $\|\cdot\|_{L}$ on $F^{u}$ is defined by $\|f\|_{L}=\bigwedge\left\{r>0,~ \exists P\subset r\Ba_{F_{+}} ~ P\uparrow \left|f\right|\right\}$. By virtue of Corollary \ref{solid2}, this seminorm is the gauge of the order adherence $\Ba_{F}^{L}$ of $\Ba_{F}$ in $F^{u}$, which is a convex solid set. Consequently, $\|\cdot\|_{L}$ is a solid seminorm (with possibly infinite values), and that $\|\cdot\|_{L}\le \|\cdot\|$ on $F$. It follows that $F$ is semicontinuous iff $\|\cdot\|_{L}=\|\cdot\|$ on $F$, demicontinuous iff $\|\cdot\|_{L}$ and $\|\cdot\|$ are equivalent on $F$, and pointwise demicontinuous iff $\|\cdot\|_{L}$ is a norm on $F$. Note that a renorming of $F$ results in an equivalent Lorentz seminorm.

Assume that $F$ is pointwise demicontinuous. According to Proposition \ref{barely} $\Ba_{F}^{L}$ is radially bounded, and so $\|\cdot\|_{L}$ is a solid norm on $F^{L}:=\bigcup\limits_{n\in\N}n\Ba_{F}^{L}$. We will call $\left(F^{L},\|\cdot\|_{L}\right)$ the \emph{Lorentz completion} of $F$. Since $\Ba_{F}^{L}$ is solid, it follows that $F^{L}$ is always an ideal in $F^{u}$, and hence order complete. The Lorentz completion in the case when $F$ is demicontinuous will be investigated in Section \ref{s6}. In the general case we have the following result.

\begin{proposition}
If $F$ is pointwise demicontinuous then $F^{L}$ is a Banach lattice.
\end{proposition}
\begin{proof}
Let $\left(g_{n}\right)_{n\in\N}\subset F^{L}_{+}$ be such that $\sum\limits_{n\in\N}\|g_{n}\|_{L}<1$. There is $\left(r_{n}\right)_{n\in\N}\subset \R_{+}$ such that $\sum\limits_{n\in\N}r_{n}\le 1$, and $\|g_{n}\|_{L}<r_{n}$, for every $n\in\N$. For every $n\in\N$ there is a directed $P_{n}\subset r_{n}\Ba_{F_{+}}$ such that $P_{n}\uparrow g_{n}$. We may assume that $0_{F}\in P_{n}$, for every $n\in\N$. Let $P:=\bigplus\limits_{n\in\N}P_{n}\subset \Ba_{F_{+}}$ (the set of finite sums), which is upwards directed, hence dominable, according to Proposition \ref{barely}. It follows that there is $g\in F^{u}$ such that $g=\bigvee P$. By definition it follows that $g\in \Ba_{F}^{L}\subset F^{L}$. It is not hard to see that $g=\bigvee\limits_{n\in\N}\sum\limits_{k=1}^{n}g_{k}$. Thus, completeness of the norm follows from  \cite[Theorem 1]{bb}.
\end{proof}

\begin{lemma}\label{noatoms}
Let $f\in F_{+}$. Then:
\item[(i)] If $0_{F}< e\le f$ is an atom, then $\|f\|_{L}\ge\|e\|$.
\item[(ii)] If $0_{F^{*}}\ne \nu\in F^{*}_{\oc}$, then $\|f\|_{L}\ge\frac{\left|\nu\left(f\right)\right|}{\|\nu\|}$.
\end{lemma}
\begin{proof}
(i): Assume that $r<\|e\|$ but $P\subset r\Ba_{F_{+}}$ is such that $P\uparrow f$. Then, for every $p\in P$ we have that $p\wedge e$ is a multiple of $e$ of the norm at most $r$, hence $p\wedge e\le \frac{r}{\|e\|}e<e$, which contradicts $P\wedge e\uparrow f\wedge e=e$.\medskip

(ii): Assume that $r>0$ is such that there is $P\subset r\Ba_{F_{+}}$ with $P\uparrow f$. Then, $\left|\nu\left(f\right)\right|\le \left|\nu\right|\left(f\right)=\bigvee\left|\nu\right|\left(P\right)\le r\|\nu\|$, and so $r\ge\frac{\left|\nu\left(f\right)\right|}{\|\nu\|}$.
\end{proof}

\begin{corollary}\label{sufficient}
If $F$ is atomic, or $F^{*}_{\oc}$ separates points of $F$, then $F$ is pointwise demicontinuous.
\end{corollary}

We will call $F$ \emph{anti-demicontinuous} if $\|\cdot\|_{L}\equiv 0$ on $F$, which is equivalent to the fact that the order adherence of $\Ba_{F}$ is equal to $F$, or in other words, for every $f\in F_{+}$ there is $P\subset\Ba_{F_{+}}$ such that $P\uparrow f$. It is easy to see that anti-demicontinuity is stable under renormings. According to Lemma \ref{noatoms} an anti-demicontinuous normed lattice cannot contain any atoms and it does not admit any nonzero order continuous norm bounded functionals. Clearly, a pointwise demicontinuous normed lattice cannot be anti-demicontinuous. We will present some examples of anti-demicontinuous AM-spaces in Examples \ref{nb} and \ref{completion} as well as Theorem \ref{cph} and Proposition \ref{faml}, but now we consider a simple example of an anti-demicontinuous normed lattice, which also exhibits the fact that anti-demicontinuity is not stable under norm completion.

\begin{example}\label{l1}
Let $F$ be $\Co\left[0,1\right]$ considered as a sublattice of $L_{1}\left[0,1\right]$. Since $F$ is dense in $L_{1}\left[0,1\right]$, it follows that the latter is the norm completion of $F$. Let us show that $F$ is anti-demicontinuous. Since $\ker\|\cdot\|_{L}$ is an ideal, it is enough to show that $\|\1\|_{L}=0$. Enumerate $\Q\cap\left[0,1\right]=\left\{q_{n}\right\}_{n\in\N}$. For $m,n\in\N$ let $f_{mn}\in\Co\left[0,1\right]$ be such that $f_{mn}\left(q_{m}\right)=1$, $\0\le f_{mn}\le \1$ and $f_{mn}$ vanish outside $\left(q_{m}-\frac{1}{n},q_{m}+\frac{1}{n}\right)$; note that $\|f_{mn}\|\le \frac{2}{m}$. For $m,n\in\N$ let $g_{mn}:=\bigvee\limits_{k=1}^{n}f_{2^{m+k},k}$; then $g_{mn}\le \1$, $g_{mn}\left(q_{k}\right)=1$, for every $k=1,...,n$, and $\|g_{mn}\|\le \sum\limits_{k=1}^{n}2^{1-m-k}\le 2^{1-m}$. It is easy to see that $2^{1-m}\Ba_{F_{+}}g_{mn}\uparrow_{n\in\N}\1$, for every $m\in\N$, thus justifying $\|\1\|_{L}=0$. Hence, $F$ is anti-demicontinuous, but its norm completion $L_{1}\left[0,1\right]$ is order continuous, and hence cannot be anti-demicontinuous.
\qed\end{example}

Let us summarize the relations between the considered properties. We have
$$\mbox{Order continuous}\Rightarrow\mbox{semi}\Rightarrow\mbox{topologically semi}\Rightarrow\mbox{demi}\Rightarrow\mbox{pointwise demi},$$
and moreover, a pointwise demicontinuous normed lattice cannot be anti-demicontinuous. It is clear that $c$ is semicontinuous but not order continuous, $c_{2}$ from Example \ref{cn} is topologically semicontinuous but not semicontinuous, while Example \ref{pcn} showcases a pointwise demicontinuous AM-space which is not demicontinuous. If $F$ is as in Example \ref{nb}, then $F\oplus_{\8}\R$ is a non-pointwise demicontinuous AM-space, which is also not anti-demicontinuous. In some cases demicontinuity is equivalent to topological semicontinuity: in the case of AM-spaces this was established in \cite[Theorem 6]{wickstead} (see also part (ii) of Theorem \ref{main}), while for Banach function spaces (i.e. ideals of $L_{0}\left(\mu\right)$, for a $\sigma$-finite measure $\mu$) this follows from the fact that the Lorentz seminorm is semicontinuous (see \cite[Theorem 1, Section 66]{zaanen}); more general results in this direction are \cite[Theorem 4]{fk} and \cite[Corollary 3.7]{cd} (see also Corollary \ref{wfqf}). However, in \cite{elliot} and \cite{att} the authors constructed a demicontinuous Banach lattice which fails to be topologically semicontinuous.

\begin{remark}
Let us say that $F$ is \emph{anti-semicontinuous} if the order closure of $\Ba_{F}$ is equal to $F$. Clearly, anti-demicontinuity implies anti-semicontinuity, and the latter is mutually exclusive with the topological semicontinuity. However, it is unknown whether it is possible for a demicontinuous normed lattice to also be anti-semicontinuous.
\qed\end{remark}

We finish the section with some examples.

\begin{example}\label{nb}
Let $K$ be a compact metric space, and assume that $\left(K_{n}\right)_{n\in\N}$ is a sequence of closed subsets of $K$ such that $K_{n}$ is a nowhere dense subset of $K_{n+1}$, for every $n\in\N$. For example, put $K:=\left[0,1\right]^{\omega}$, and let $K_{n}$ be the set of all points of $K$, where all coordinates starting with $n$-th vanish.

For $n\in\N$ let $E_{n}:=\left(\Co\left(K_{n}\right),\frac{1}{2^{n}}\|\cdot\|_{\8}\right)$, and let $E:=\bigoplus_{\8}\left\{E_{n},~ n\in\N\right\}$. Let $J:\Co\left(K\right)\to E$ be defined by $Jf:=\left(\left.f\right|_{K_{n}}\right)_{n\in\N}$. Let $F:=\overline{J\Co\left(K\right)}$ in $E$. We will denote the constant zero and one functions as $\0$ and $\1$.

For every $m\ge n$ let $T_{n}:E\to E_{n}$ be the coordinate projection, and $R_{mn}:E_{m}\to E_{n}$ be the restriction operator. Then for every $e\in J\Co\left(K\right)$ and $m\ge n$ we have $R_{mn}T_{m}e=T_{n}e$. By continuity, it follows that $R_{mn}T_{m}=T_{n}$ on $F$.

We claim that $\|J\1\|_{L}=0$. Fix $n\in\N$. Let $g:=d\left(\cdot,K_{n}\right)\in\Co\left(K\right)$ and let $h_{k}:=\1\wedge kg$. Clearly, $\left(h_{k}\right)_{k\in\N}\subset\left[\0,\1\right]$ is an increasing sequence, and so $\left(Jh_{k}\right)_{k\in\N}\subset\left[0_{F},J\1\right]$ is also an increasing sequence. For every $k\in\N$ and $m\le n$ we have $\left.h_{k}\right|_{K_{m}}=\0$, and so $T_{m}Jh_{k}=\0$. It follows that $\|Jh_{k}\|=\bigvee\limits_{m>n}\frac{1}{2^{m}}\|T_{m}Jh_{k}\|_{\8}\le \frac{1}{2^{n+1}}$.

Let us prove that $J\1=\bigvee\limits_{k\in\N}Jh_{k}$ in $F$. Let $e\in F$ be an upper bound for $\left(Jh_{k}\right)_{k\in\N}$. We need to show that $e\ge J\1$, which is equivalent to $T_{m}e\ge \1_{K_{m}}$, for every $m\in \N$. First, assume that $m>n$. For $x\in K_{m}\backslash K_{n}$ we have $g\left(x\right)>0$, and so $h_{k}\left(x\right)\uparrow 1$. Hence, $T_{m}e\ge \1_{K_{m}\backslash K_{n}}$, but since $K_{n}$ is nowhere dense in $K_{m}$, and $T_{m}e$ is continuous, we conclude that $T_{m}e\ge\1_{K_{m}}$. If $m\le n$, then $T_{m}e=R_{n+1,m}T_{n+1}e\ge R_{n+1,m}\1_{K_{n+1}}=\1_{K_{m}}$.

It follows that for every $n\in\N$ there is a sequence in $\frac{1}{2^{n+1}}\Ba_{F}$ which increases to $J\1$, and so $\|J\1\|_{L}=0$. Note that since $\|\cdot\|_{L}\le \|\cdot\|$ is a solid seminorm on $F$, its kernel is a closed ideal. Since $J\1$ is a strong unit of $J\Co\left(K\right)$ and the latter is dense in $F$, we conclude that $\|\cdot\|_{L}\equiv 0$.
\qed\end{example}

\begin{example}\label{completion}
Let $X$ be the space of all finite sequences of natural numbers. We will denote the concatenation of $x$ and $y$ by $x^{\frown} y$, and the fact that $x$ is an initial segment of $y$ -- by $x\prec y$.

Let $F:=\left\{f\in\ell_{\8}\left(X\right),~ \lim\limits_{n\to\8}f\left(x^{\frown}n\right)=\frac{1}{2}f\left(x\right)\right\}$, which is clearly a sublattice of $\ell_{\8}\left(X\right)$. It is not hard to check that $F$ is closed in $\ell_{\8}\left(X\right)$, hence it is an AM-space. Let us show that $F$ is anti-demicontinuous.

For $x\in X$ let $e_{x}:X\to\left[0,1\right]$ be defined by $e_{x}\left(y\right)=\frac{1}{2^{\left|y\right|-\left|x\right|}}$, if $x\prec y$ and $e_{x}\left(y\right)=0$, otherwise. By considering the case when $x\prec y$ and $x\not\prec y$ one can easily show that $e_{x}\in F$, with $\|e_{x}\|=1$, for every $x\in X$.

Let $f\in F_{+}$ and let $Y:=f^{-1}\left[0,1\right)$. We will show that $f=\bigvee_{y\in Y}\left(f\wedge e_{y}\right)$. Clearly, $f$ is an upper bound for the set in question. Let $g\in F$ be such that $g\ge f\wedge e_{y}$, for every $y\in Y$. We will prove by induction over $n\in\N_{0}$ that if $x\in X$ is such that $f\left(x\right)<2^{n}$, then $f\left(x\right)\le g\left(x\right)$. If $n=0$ and $f\left(x\right)<2^{n}=1$, then $x\in Y$, and so $g\left(x\right)\ge f\left(x\right)\wedge e_{x}\left(x\right)=f\left(x\right)\wedge 1=f\left(x\right)$. Assume that the claim is proven for $n$, and let $x\in X$ be such that $f\left(x\right)<2^{n+1}$. Then, $\lim\limits_{m\to\8}f\left(x^{\frown}m\right)=\frac{1}{2}f\left(x\right)<2^{n}$, and so for large enough $m\in\N$ we have $f\left(x^{\frown}m\right)<2^{n}$, which implies $f\left(x^{\frown}m\right)\le g\left(x^{\frown}m\right)$, by the hypothesis of induction, and so $f\left(x\right)=2\lim\limits_{m\to\8}f\left(x^{\frown}m\right)\le 2\lim\limits_{m\to\8}g\left(x^{\frown}m\right)=g\left(x\right)$. We conclude that $f\le g$, and so $f=\bigvee_{y\in Y}\left(f\wedge e_{y}\right)$. As $f\wedge e_{y}\in\Ba_{F_{+}}$, for every $y\in Y$, we have established that $F$ is anti-demicontinuous.\medskip

Now let $G:=\ell_{1}\left(X\times\N,w\right)\oplus_{1} F$, where $w:X\times\N\to\left(0,1\right]$ is a summable weight. Since pointwise demicontinuity passes down to projection bands, it follows that $G$ does not have this property. Let $H:=\left\{\left(e,f\right)\in G,~ \forall x\in X:~ \lim\limits_{n\to\8}e\left(x,n\right)=f\left(x\right)\right\}$, which is clearly a sublattice of $G$.

Let $0_{G}<\left(e,f\right)\in H$. If $e=\0$, then $f\left(x\right)=\lim\limits_{n\to\8}e\left(x,n\right)=0$, for every $x\in X$, implying $f=\0$, contradiction. Hence, there are $x_{0}\in X$ and $n_{0}\in\N$ such that $r:=e\left(x_{0},n_{0}\right)>0$. Then, $r\1_{\left\{\left(x_{0},n_{0}\right)\right\}}\le e$, and it is easy to see that $\left(r\1_{\left\{\left(x_{0},n_{0}\right)\right\}},\0\right)$ is an atom in $H$. It follows that $H$ is atomic, hence by Corollary \ref{sufficient} it is pointwise demicontinuous.

Let us show that $H$ is dense in $G$. Let $e\in \ell_{1}\left(X\times\N,w\right)$, let $f\in F$ and let $\varepsilon>0$. There is a finite $A\subset X\times\N$ such that $\sum\limits_{\left(x,n\right)\notin A}\left|e\left(x,n\right)\right|w\left(x,n\right)<\frac{\varepsilon}{2}$ and $\sum\limits_{\left(x,n\right)\notin A}w\left(x,n\right)<\frac{\varepsilon}{2\|f\|+1}$. Consider $e':X\times\N\to\R$ which agrees with $e$ on $A$ and such that $e'\left(x,n\right)=f\left(x\right)$, when $\left(x,n\right)\notin A$. We have that $$\|e'-e\|=\sum\limits_{\left(x,n\right)\notin A}\left|e\left(x,n\right)-f\left(x\right)\right|w\left(x,n\right)\le\frac{\varepsilon}{2}+\frac{\varepsilon}{2\|f\|+1}\|f\|\le\varepsilon,$$ and so in particular $e'\in \ell_{1}\left(X\times\N,w\right)$. Since $A$ is finite, for every $x\in X$ and large enough $n$, we have $\left(x,n\right)\notin A$, hence $e'\left(x,n\right)=f\left(x\right)$, and so $\left(e',f\right)\in H$ at the distance at most $\varepsilon$ from $\left(e,f\right)$. We conclude that $H$ is dense in $G$.
\qed\end{example}

\section{Monotone boundedness and completeness}

In this section $F$ is a normed lattice. We call it \emph{monotonically bounded} if every directed subset of $\Ba_{F_{+}}$ is order bounded. Moreover, we say that $F$ is \emph{$r$-monotonically bounded} if every directed subset of $\Ba_{F_{+}}$ has an upper bound in $r\Ba_{F}$, and \emph{$r+$-monotonically bounded} if it is $s$-monotonically bounded, for every $s>r$. It is clear that $\Co\left(K\right)$ is $1$-monotonically bounded, for every compact Hausdorff $K$. Monotone boundedness passes down to projection bands and majorizing sublattices, and is preserved under renormings. This condition has been considered recently in e.g. \cite{abt} under the name ``strong Nakano property'' and in \cite{aat} under the name ``b-property''. A generalization of monotone boundedness to topological vector lattices was studied in e.g. \cite{taylor}, \cite{ema} (and the references therein) under the name ``BOB''.

\begin{theorem}\label{lev}
The following conditions are equivalent:
\item[(i)] $F$ is monotonically bounded;
\item[(ii)] $F$ is $r$-monotonically bounded, for some $r\ge 1$;
\item[(iii)] Every directed dominable subset of $\Ba_{F_{+}}$ is order bounded.
\end{theorem}
\begin{proof}
(ii)$\Rightarrow$(i)$\Rightarrow$(iii) is trivial. (i)$\Rightarrow$(ii): Assume towards contradiction that for every $n\in\N$ there is an upward directed $P_{n}\subset \Ba_{F_{+}}$ which has no upper bounds of norm less than $n2^{n}$. We may assume that $0_{F}\in P_{n}$, for every $n\in\N$. Let $P:=\bigplus\limits_{n\in\N}\frac{1}{2^{n}}P_{n}$ (the set of finite sums). It is easy to show that $P$ is an upward directed subset of $\Ba_{F_{+}}$, and so it has an upper bound, say $f$. Let $n\in\N$ be such that $n>\|f\|$. Then, $f\ge P\supset \frac{1}{2^{n}}P_{n}$ yields $2^{n}f\ge P_{n}$. By assumption this implies that $\|2^{n}f\|\ge n2^{n}$, hence $\|f\|\ge n$, contradiction.\medskip

(iii)$\Rightarrow$(i): By Proposition \ref{barely} it is enough to show that $F$ is pointwise demicontinuous. Assume towards contradiction that there is $f>0_{F}$ such that for every $s>0$ there is $P\subset s\Ba_{F_{+}}$ with $P\uparrow f$. According to Lemma \ref{noatoms} $\left[0_{F},f\right]$ contains no atoms. It is not difficult to inductively construct a disjoint sequence $\left(d_{n}\right)_{n\in\N}\subset\left(0_{F},f\right]$. For every $n\in\N$ there is $P_{n}\subset \frac{\|d_{n}\|}{n2^{n}}\Ba_{F_{+}}$ such that $P_{n}\uparrow f$, and so $Q_{n}:=P_{n}\wedge d_{n}\uparrow f\wedge d_{n}=d_{n}$. We may assume that $0_{F}\in P_{n}$, hence $0_{F}\in Q_{n}$, for every $n\in\N$. Let $Q:=\bigplus\limits_{n\in\N}\frac{n}{\|d_{n}\|}Q_{n}$, which is a directed subset of $\Ba_{F_{+}}$. We claim that $Q$ is dominable. Since $\left(\frac{n}{\|d_{n}\|}d_{n}\right)_{n\in\N}$ is disjoint, it has a supremum, say $g$, in $F^{u}$. Every element of $Q$ is of the form $$\frac{1}{\|d_{1}\|}\left(p_{1}\wedge d_{1}\right)+...+\frac{n}{\|d_{n}\|}\left(p_{n}\wedge d_{n}\right)\le \frac{1}{\|d_{1}\|}d_{1}+...+\frac{n}{\|d_{n}\|}d_{n}= \frac{1}{\|d_{1}\|}d_{1}\vee...\vee\frac{n}{\|d_{n}\|}d_{n}\le g,$$ where $p_{k}\in P_{k}$, for every $k=1,...,n$, and the equality in the middle follows from disjointness of $d_{1},...,d_{n}$. Hence, $Q\le g$, and so $Q$ is dominable. It follows that $Q$ has an upper bound, say $h$ in $F$. Take $n>\|h\|$, and note that $h\ge Q\supset \frac{n}{\|d_{n}\|}Q_{n}$ and $Q_{n}\uparrow d_{n}$ imply $h\ge \frac{n}{\|d_{n}\|}d_{n}$, so that $n>\|h\|\ge n$, contradiction.
\end{proof}

\begin{proposition}\label{lef}
Monotone boundedness implies demicontinuity. The $r+$-monotone boundedness implies $r$-demicontinuity.
\end{proposition}
\begin{proof}
We prove the second claim first. Assume that $P\subset\Ba_{F_{+}}$ and $f\in F$ are such that $P\uparrow f$. For every $s>r$ there is $g\in s\Ba_{F}$ such that $P\le g$. Then, $0_{F}\le f\le g$, and so $\|f\|\le \|g\|\le s$. Since $s$ was arbitrary, we conclude that $\|f\|\le r$. The first claim now follows from Theorem \ref{lev}.
\end{proof}

$F$ is called \emph{($r$-)monotonically complete} if every directed subset of $\Ba_{F_{+}}$ has a supremum in $F$ (of norm at most $r$). Every monotonically complete normed lattice is norm complete (see \cite[Proposition 2.4.19(i)]{mn}; it also follows immediately from \cite[Theorem 1]{bb}).

Clearly, $F$ is ($r$-)monotonically complete if and only if it is ($r$-)monotonically bounded and order complete. Moreover, $F$ is $r$-monotonically complete if and only if it is monotonically complete and $r$-demicontinuous. It is easy to see that ($r$-)monotone completeness and boundedness pass down to projection bands in $F$. Monotone completeness is stable under renormings.

Note that $F^{*}$ is always $1$-monotonically complete: if $P\subset\Ba_{F_{+}^{*}}$ is directed, it has a weak* limit in $\Ba_{F^{*}}$, which can be easily shown to be the supremum of $P$, using \cite[Corollary 1.3.4(ii)]{mn}. Since $F^{*}$ is an ideal in $F^{\sim}$, while $F^{\sim}_{\oc}$ is a projection band there (see \cite[Theorems 1.73 and 2.22]{ab0}), it follows that $F^{*}_{\oc}=F^{\sim}_{\oc}\cap F^{*}$ is a projection band in $F^{*}$, and so it is $1$-monotonically complete.

The following results are known (see \cite[Proposition 1]{ema} and \cite[Lemma 2.8]{abt}), but we prove them here for completeness. Recall that if $A\subset F$, the set of all suprema of finite subsets of $A$ is denoted by $A^{\vee}$; note that the latter set is directed.

\begin{corollary}\label{levi}
\item[(i)] $Q\subset F$ is order bounded in $F^{**}$ if and only if $\left|Q\right|^{\vee}$ is norm bounded in $F$.
\item[(ii)] $F$ is monotonically bounded if and only if every subset of $F$ which is order bounded in $F^{**}$ is also order bounded in $F$.
\end{corollary}
\begin{proof}
(i): Necessity: If $g\in F^{**}_{+}$ is such that $Q\subset\left[-g,g\right]$, then $\left|Q\right|^{\vee}\subset\left[-g,g\right]\subset\|g\|\Ba_{F^{**}}$, and so $\left|Q\right|^{\vee}$ is norm bounded.

Sufficiency: Since $F^{**}$ is monotonically complete, and $\left|Q\right|^{\vee}$ is norm bounded and directed, there is $g\in F^{**}$ such that $\left|Q\right|^{\vee}\uparrow g$, in particular $Q\subset\left[-g,g\right]$.\medskip

(ii): Necessity: If $Q\subset F$ is bounded in $F^{**}$, then $\left|Q\right|^{\vee}$ is norm bounded and directed, hence order bounded in $F$, due to monotone boundedness. It then follows that $Q$ is order bounded in $F$ as well.\medskip

Sufficiency: If $P\subset\Ba_{F_{+}}\subset \Ba_{F_{+}^{**}}$ is directed, then it is order bounded in $F^{**}$, hence in $F$.
\end{proof}

\begin{proposition}\label{pedro}
If $F$ has an order dense sublattice $E$ which is ($r$-)monotonically bounded, then $F$ is ($r$-)monotonically bounded and $E$ is majorizing in $F$.
\end{proposition}
\begin{proof}
According to Theorem \ref{lev}, it is enough to prove the version with a constant. Let $P\subset\Ba_{F_{+}}$ be directed. Then, $Q:= \bigcup\limits_{p\in P}\left[0_{F},p\right]_{E}$ is directed, and $p\wedge Q\uparrow p$, for every $p\in P$ (see arguments from the proof of Lemma \ref{solid}). Since also $Q\subset\Ba_{E_{+}}$, it has an upper bound in $E$, say $e$ with $\|e\|\le r$. Then, $p=\bigvee\left(Q\wedge p\right)\le e$, for every $p\in P$. Thus, $P$ is order bounded in $F$ of norm at most $r$. The fact that $E$ is majorizing in $F$ follows from applying the same argument to $P:=\left\{f\right\}$, for $f\in\Ba_{F_{+}}$.
\end{proof}

Let us present a simplified proof of \cite[Theorem 2.3]{aw} and \cite[Theorem 2.6]{taylor} and provide their quantitative versions.

\begin{theorem}\label{awt}
The following conditions are equivalent:
\item[(i)] $F$ is monotonically complete;
\item[(ii)] Every directed norm bounded dominable subset of $F$ has a supremum in $F$;
\item[(iii)] $F$ is norm complete and every disjoint $D\subset F_{+}$ such that $D^{\vee}\subset\Ba_{F}$ has a supremum in $F$;
\item[(iv)] The order adherence of $\Ba_{F}$ in $F^{u}$ is contained in $F$;
\item[(v)] Every uo-Cauchy net in $\Ba_{F}$ has a uo-limit.
\end{theorem}
\begin{proof}
(i)$\Rightarrow$(ii) is trivial. For the converse we first observe that every order bounded directed set is norm bounded and dominable, hence has a supremum. It follows that $F$ is order complete. It is also monotonically bounded, according to Theorem \ref{lev}. Hence, $F$ is monotonically complete.\medskip

As was mentioned above, monotone completeness implies norm completeness. If $D\subset F_{+}$ is such that $D^{\vee}\subset\Ba_{F}$, then monotone completeness implies that $D^{\vee}$ has a supremum in $F$, which is also the supremum of $D$. This way we have established (i)$\Rightarrow$(iii).\medskip

(iii)$\Rightarrow$(ii): First, since $F$ is norm complete, it is relatively uniformly complete (see \cite[Proposition 1.18(iv)]{mn}). Moreover, if $D\subset F_{+}$ is disjoint and order bounded, then $D^{\vee}$ is norm bounded, and so $D$ has a supremum. By Veksler-Geiler theorem (see e.g. \cite[Corollary 3.11]{erz2}) this latter property of $F$ together with relative uniform completeness implies order completeness.

Let $P\subset\Ba_{F_{+}}$ be directed and dominable, hence there is $f\in F^{u}$ such that $P\uparrow f$. It follows from Proposition \ref{eps} that there is a disjoint $D\subset F_{+}$ such that $\bigvee_{F^{u}} D=\frac{1}{2}f$ and for every $d\in D$ there is $p\in P$ such that $d\le p$. Let us show that $D^{\vee}\subset\Ba_{F}$. Let $d\in D^{\vee}$ and let $d_{1},...,d_{n}$ be such that $d=d_{1}\vee...\vee d_{n}$. There are $p_{1},...,p_{n}\in P$ such that $d_{k}\le p_{k}$, for every $k=1,...,n$, and $p\in P$ such that $p\ge p_{1},...,p_{n}$. Hence, $d\le p\in P\subset\Ba_{F_{+}}$, which yields $d\in\Ba_{F}$. By our assumption $D$ has a supremum in $F$ which by regularity must be $\frac{1}{2}f$, thus $f\in F$.\medskip

(ii)$\Rightarrow$(iv): As was already discussed before, $F$ is order complete, hence ideal in $F^{u}$, and so $\Ba_{F}$ is solid in $F^{u}$. According to Corollary \ref{solid1}, if $f\in F^{u}$ is contained in the order adherence of $\Ba_{F}$ in $F^{u}$, then there is a directed $P\subset \Ba_{F_{+}}$ such that $P\uparrow\left|f\right|$. Then, $P$ is directed norm bounded and dominable, hence has a supremum in $F$, which must be $\left|f\right|$. Thus, $\left|f\right|\in F$, and since $F$ is an ideal in $F^{u}$, it follows that $f\in F$.\medskip

(iv)$\Rightarrow$(v): First, we again have to start with establishing order completeness of $F$. To that end, note that for every $g\in F^{\delta}_{+}$ we have that $\left[0_{F},g\right]_{F}$ is order bounded in $F$, hence norm bounded, and $\left[0_{F},g\right]_{F}\uparrow g$, hence $\left[0_{F},g\right]_{F}\too g$, which implies that $g\in F$. Thus, $F=F^{\delta}$, hence $F$ is order complete, and so $\Ba_{F}$ is solid in $F^{u}$.

Let $\left(f_{p}\right)_{p\in P}\subset\Ba_{F}$ be uo-Cauchy. Recall that in $F^{u}$ every uo-Cauchy net has a uo-limit (see \cite[Theorem 17]{azouzi}). Hence, there is $f\in F^{u}$ such that $f_{p}\touo f$. Then $f$ is contained in the uo adherence of $\Ba_{F}$ in $F^{u}$. According to Corollary \ref{solid1} this adherence coincides with the order adherence, and so by our assumption it is contained in $F$. Since uo convergence of $F$ is the restriction of uo convergence on $F^{u}$ (see e.g. \cite[Theorem 5.3]{erz}), we conclude that $f_{p}\touo f$ in $F$.\medskip

(v)$\Rightarrow$(ii): Let $P\subset\Ba_{F_{+}}$ be directed and dominable, hence there is $f\in F^{u}$ such that $P\uparrow f$. Then, $P\touo f$, and so $P$ is a uo-Cauchy net in $F^{u}$, thus in $F$. By our assumption this net has a uo-limit in $F$, which must be equal to $f$, because the embedding of $F$ into $F^{u}$ is uo-continuous (see e.g. \cite[Theorem 5.2]{erz}). Hence, $f\in F$.
\end{proof}

Similarities between Theorems \ref{lev} and \ref{awt} inspire the following question, which has a positive answer under the SMP (see Proposition \ref{eps} and Remark \ref{smp}).

\begin{question}
Assume that $F$ is demicontinuous and such that every disjoint $D\subset F_{+}$ such that $D^{\vee}\subset\Ba_{F}$ is order bounded in $F$. Is $F$ monotonically bounded?
\end{question}

\begin{theorem}\label{rawt}
The following conditions are equivalent:
\item[(i)] $F$ is $r$-monotonically complete;
\item[(ii)] Every directed dominable subset of $\Ba_{F_{+}}$ has a supremum in $F$ of norm at most $r$;
\item[(iii)] $F$ is norm complete and every disjoint $D\subset F_{+}$ such that $D^{\vee}\subset\Ba_{F}$ has a supremum of norm at most $r$;
\item[(iv)] The order adherence of $\Ba_{F}$ in $F^{u}$ is contained in $r\Ba_{F}$;
\item[(v)] Every uo-Cauchy net in $\Ba_{F}$ has a uo-limit in $r\Ba_{F}$.
\end{theorem}
\begin{proof}
Most of the implications are proven similarly to the corresponding implications in Theorem \ref{awt}. Let us only revise (iii)$\Rightarrow$(i): First, it follows from Theorem \ref{awt} that $F$ is monotonically complete. Let $P\subset\Ba_{F_{+}}$ be directed, hence there is $f\in F$ such that $P\uparrow f$. For every $0< s<1$ there is a disjoint $D\subset F_{+}$ such that $\bigvee D=sf$ and for every $d\in D$ there is $p\in P$ such that $d\le p$. Arguing as in the proof of the implication (iii)$\Rightarrow$(ii) of Theorem \ref{awt} we see that $D^{\vee}\subset \Ba_{F}$, and so $\|sf\|\le r$, thus $\|f\|\le \frac{r}{s}$. Since $s$ was arbitrary, we conclude that $\|f\|\le r$.
\end{proof}

In the special case of $r=1$, the conditions (i) and (v) in Theorem \ref{rawt} can be restated as follows.

\begin{corollary}
$F$ is $1$-monotonically complete if and only if $\Ba_{F}$ is complete with respect to uo convergence.
\end{corollary}

\section{The Lorentz completion}\label{s6}

Let us review some concepts introduced in Section \ref{s4}. Recall that if $F$ is a normed lattice, the Lorentz seminorm $\|\cdot\|_{L}$ on $F^{u}$ is defined by $\|f\|_{L}=\bigwedge\left\{r>0,~ \exists P\subset r\Ba_{F_{+}} ~ P\uparrow \left|f\right|\right\}$. This is a solid seminorm (with possibly infinite values) equal to the gauge of the order adherence $\Ba_{F}^{L}$ of $\Ba_{F}$ in $F^{u}$. We have $\|\cdot\|_{L}$ finitely valued on $F^{L}:=\bigcup\limits_{n\in\N}n\Ba_{F}^{L}$, and $\|\cdot\|_{L}\le \|\cdot\|$ on $F$. If $F$ is pointwise demicontinuous then $\|\cdot\|_{L}$ is a complete norm. We then call $\left(F^{L},\|\cdot\|_{L}\right)$ the Lorentz completion of $F$. It follows that $F^{L}$ is pointwise demicontinuous, order complete and contains $F$ as an order dense sublattice. If $F$ is $r$-demicontinuous then $\|\cdot\|\le r\|\cdot\|_{L}$. In this case the Lorentz completion has a number of additional properties.

\begin{theorem}\label{delta}
Assume that $F$ is $r$-demicontinuous. Then:
\item[(i)] $F^{L}$ is an $r$-monotonically complete Banach lattice.
\item[(ii)] The norm of $F$ has an $r^{2}$-demicontinuous extension to $F^{L}$, equivalent to $\|\cdot\|_{L}$.
\item[(iii)] The norm completion of $F$ is $r^{2}$-demicontinuous.
\item[(iv)] If $T:F\to E$ is a contractive order continuous homomorphism into a demicontinuous normed lattice, then there is a unique order continuous extension $T^{L}:F^{L}\to E^{L}$ of $T$. Moreover, $T^{L}$ is a contractive homomorphism.
\item[(v)] If $G$ is a normed lattice which isomorphically contains $F$ as an order dense sublattice, then there is a unique injective order continuous homomorphism $J:G\to F^{L}$ which leaves $F$ in place.
\item[(vi)] If $F$ is $r$-semicontinuous, then so is $F^{L}$. In particular, the Lorentz completion preserves topological semicontinuity.
\end{theorem}
\begin{proof}
(i): First, $\Ba_{F}^{L}\cap F$ is the order adherence of $\Ba_{F}$ in $F$, and so $\Ba_{F}^{L}\cap F\subset r\Ba_{F}$, by virtue of Proposition \ref{rwf}. Let $C:=\bigcap\limits_{s>1}s\Ba_{F}^{L}$, which is the unit ball with respect to $\|\cdot\|_{L}$. By Lemma \ref{solid}, for every $s>1$ we have that the order adherence of $s\Ba_{F}^{L}$ in $F^{u}$ coincides with the solid hull of the order adherence of $s\Ba_{F}^{L}\cap F\subset rs\Ba_{F}$, and so is contained in $rs\Ba_{F}^{L}$. Thus, the order adherence of $C$ in $F^{u}$ is contained in $\bigcap\limits_{s>1}rs\Ba_{F}^{L}=rC$. It now follows from Theorem \ref{rawt} that $F^{L}$ is $r$-monotonically complete.\medskip

(ii): It follows from Lemma \ref{conve} that $\Sol\Ba_{F}$ in $F^{L}$ is convex and its intersection with $F$ recovers $\Ba_{F}$. Therefore, the gauge of $\Sol\Ba_{F}$ extends $\|\cdot\|$. Recall that $\frac{1}{r}\Ba_{F}^{L}\cap F\subset \Ba_{F}\subset \Ba_{F}^{L}$. Let $B$ be the convex hull of $\frac{1}{r}\Ba_{F}^{L}\cup\Sol\Ba_{F}$. Note that $B$ is solid as a convex hull of a solid set. Then, the gauge $\|\cdot\|_{B}$ of $B$ is a solid norm which extends $\|\cdot\|$ (see \cite[Proposition 2.14]{fhhmz}), and since $\frac{1}{r}\Ba_{F}^{L}\subset B\subset \Ba_{F}^{L}$, it follows that $r\|\cdot\|_{L}\ge\|\cdot\|_{B}\ge\|\cdot\|_{L}$. As $\|\cdot\|_{L}$ is $r$-demicontinuous by (i) and Proposition \ref{lef}, and $P\uparrow f$ in $F^{L}_{+}$, then $\|f\|_{B}\le r\|f\|_{L}\le r^{2}\sup\limits_{p\in P}\|p\|_{L}\le r^{2}\sup\limits_{p\in P}\|p\|_{B}$.\medskip

(iii): It follows from (i) and (ii) that $\left(F^{L},\|\cdot\|_{B}\right)$ is a monotonically complete, hence norm complete normed lattice which contains $F$ isometrically as an order dense sublattice. It follows that the closure of $F$ in $F^{L}$ is also an order dense, hence regular sublattice which is isometrically isomorphic to the norm completion of $F$. Since the $r^{2}$-demicontinuity passes down to regular sublattices, the claim follows.\medskip

(iv): Let $T^{u}:F^{u}\to E^{u}$ be the order continuous homomorphism which extends $T$ (combine \cite[Theorems 1.84 and 7.17]{ab0}). Every $0_{F^{u}}\le g\in \Ba_{F}^{L}$ is a supremum of a directed $P\subset\Ba_{F_{+}}$. Then, $TP$ is a directed subset of $\Ba_{E_{+}}$ with $TP\uparrow T^{u}g$, implying that $T^{u}g\in \Ba_{E}^{L}$. It is now easy to deduce that $T^{L}:=\left.T^{u}\right|_{F^{L}}$ is a contractive order continuous homomorphism. If $S:F^{L}\to E^{L}$ is another order continuous extension of $T$, then $S-T^{L}$ is order continuous and vanishes on $F$. As the order adherence of $F$ in $F^{L}$ is $F^{L}$, we conclude that $S=T^{L}$.\medskip

(v): Let $T:F\to G$ be an isomorphic embedding such that $TF$ is order dense in $G$. Since $G$ is order dense in $G^{u}$, it follows that $TF$ is order dense in $G^{u}$, and so there is an order isomorphism $T^{u}:F^{u}\to G^{u}$, extending $T$ (see \cite[Theorem 7.19]{ab0}). It is enough to show that $T^{u}F^{L}\supset G$, since then $J:=\left.T^{u-1}\right|_{G}$ does the job. Uniqueness is proven in the same way as in (iv).

Let $g\in G_{+}$ and consider the set $P:=T^{-1}\left[0_{G},g\right]$. Clearly, $P$ is directed in $F$ and since $T$ is an isomorphic embedding, and $\left[0_{G},g\right]$ is norm bounded in $G$, it follows that $P$ is norm bounded in $F$, hence in $F^{L}$. Since by (i) the latter is monotonically complete, there is $f\in F^{L}\subset F^{u}$ such that $P\uparrow f$. Due to order denseness of $TF$ in $G$, it follows that $\bigvee TP= \bigvee\left[0_{G},g\right]_{TF}=g$. As $T^{u}$ is order continuous, we conclude that $T^{u}f=\bigvee TP=g$, and so $g\in T^{u}F^{L}$.\medskip

(vi): Let $\|\cdot\|'$ be the equivalent semicontinuous norm on $F$ satisfying $\frac{1}{r}\|\cdot\|\le \|\cdot\|'\le \|\cdot\|$, and let $B$ be its unit ball. We have that $\Ba_{F}\subset B\subset r\Ba_{F}$, hence $\Ba_{F}^{L}\subset B^{L}\subset r\Ba_{F}^{L}$, where $B^{L}$ is the order adherence of $B$ in $F^{u}$, i.e. the unit ball of the Lorentz completion of $\left(F,\|\cdot\|'\right)$. By virtue of (i), the norm $\|\cdot\|'_{L}$ of latter completion is semicontinuous. Since it satisfies $\frac{1}{r}\|\cdot\|_{L}\le \|\cdot\|'_{L}\le \|\cdot\|_{L}$, the result follows.
\end{proof}

\begin{question}
Can the constants in Proposition \ref{dua} and parts (ii) and (iii) of Theorem \ref{delta} be improved?
\end{question}

To summarize, monotone completeness plays a role in the Banach lattice theory somewhat analogous to the role universally complete vector lattices play in the vector lattice theory, and the Lorentz completion may be viewed as the analogue of the universal completion.

It follows from part (iv) of Theorem \ref{delta} that if $\mathbf{DCNL}$ is a category of demicontinuous normed lattices with order continuous homomorphisms and $\mathbf{MCBL}$ is its full subcategory whose objects are monotone complete Banach lattices, then $F^{L}$ is the reflection of $F$ in $\mathbf{MCBL}$.

\begin{question}
Suppose $F$ is a normed lattice, $E$ is a monotone complete Banach lattice, and $J:F\to E$ is an order-continuous homomorphism such that whenever $T:F\to G$ is an order-continuous homomorphism into a monotone complete Banach lattice $G$, there is a unique order-continuous homomorphism $\widehat{T}:E\to G$ such that $\widehat{T}J=T$. What can be said about $F$, $E$, and $J$?
\end{question}

\begin{example}
A monotone complete Banach lattice may not be topologically semicontinuous. Indeed, let $F$ be a demicontinuous Banach lattice which is not topologically semicontinuous, e.g. the example from \cite{att}. Then, $F^{L}$ is monotone complete and contains $F$ isomorphically as an order dense sublattice. If $F^{L}$ was topologically semicontinuous, the same would be true about $F$, contradiction.
\qed\end{example}

Every $F$ isometrically embeds into $F^{**}$ which is monotonically complete. On the other hand, in order to get an order continuous embedding $F$ has to be demicontinuous, as the following result shows.

\begin{corollary}\label{lorentz}
The following conditions are equivalent:
\item[(i)] $F$ is demicontinuous;
\item[(ii)] $F$ embeds isometrically as an order dense sublattice of a monotonically complete Banach lattice;
\item[(iii)] $F$ embeds isomorphically as a regular sublattice of a monotonically complete Banach lattice.\medskip

Moreover, the monotonically complete Banach lattice which contains $F$ isomorphically as an order dense sublattice is unique up to an isomorphism.
\end{corollary}
\begin{proof}
(i)$\Rightarrow$(ii) follows from parts (i) and (ii) of Theorem \ref{delta}. (ii)$\Rightarrow$(iii) is trivial.

(iii)$\Rightarrow$(i) follows from the fact that a monotonically complete Banach lattice is monotonically bounded, hence demicontinuous, by virtue of Proposition \ref{lef}, and the latter property passes down to regular sublattices, according to Proposition \ref{pass}.

Uniqueness follows from part (v) of Theorem \ref{delta}, and the fact that if $G$ is an order dense sublattice of a normed lattice $H$ (with equivalent norms), and $G$ is monotonically complete, then $G=H$. Indeed, for every $h\in H_{+}$ we have that $h=\bigvee\left[0_{H},h\right]{G}$, but $\left[0_{H},h\right]_{G}$ is a directed norm bounded subset of $G$, and so it has a supremum in $G$, which must be equal to $h$, thereby showing that $h\in G$.
\end{proof}

\begin{corollary}
If $E\subset F$ is an order dense sublattice, then $E^{L}$ is isomorphic to $F^{L}$.
\end{corollary}

We denote $F^{**}_{\oc}:=\left(F^{*}_{\oc}\right)^{*}_{\oc}=\left(F^{*}_{\oc}\right)^{\sim}_{\oc}$. Note that it is a $1$-monotonically complete Banach lattice.

\begin{corollary}\label{seco}
Assume that $F$ is $r$-demicontinuous and that $F^{*}_{\oc}$ separates points of $F$. Then, $F^{**}_{\oc}=F^{L}$ (isomorphically).
\end{corollary}
\begin{proof}
By Proposition \ref{dua} we have that $F^{*}_{\oc}$ is $r^{2}$-norming, and so the usual embedding of $F$ into $F^{**}_{\oc}$ is bounded below. Also, $F$ is order dense in $F^{**}_{\oc}$ (see \cite[Theorem 1.83(3)]{ab0}), and so $F^{**}_{\oc}$ is a monotonically complete Banach lattice which contains $F$ isomorphically as an order dense sublattice. The result now follows from Corollary \ref{lorentz}.
\end{proof}

\begin{remark}
Note that if $F$ is demicontinuous then $F^{LL}=F^{L}$, as a set, but the norms might be non-equal (albeit equivalent). In fact, arguing similarly to Proposition \ref{dua} one can show that in the context of Corollary \ref{seco} we have $F^{**}_{\oc}=F^{LL}$ (isometrically).
\qed\end{remark}

Recall that the norm of $F$ can be extended to a solid norm on $F^{\delta}$, in a non-unique way. The largest of these extensions is given by $\|g\|_{\delta}:=\bigwedge\left\{\|f\|,~ f\ge \left|g\right|\right\}$, for $g\in F^{\delta}$, which is the gauge of $\Sol \Ba_{F}$ in $F^{\delta}$. The following result continues the line of Proposition \ref{pedro}.

\begin{proposition}\label{delta1}
\item[(i)] $F^{\delta}$ is the ideal in $F^{L}$ generated by $F$ with $\|\cdot\|_{L}\le \|\cdot\|_{\delta}$.
\item[(ii)] $F$ is monotonically bounded if and only if $F^{L}=F^{\delta}$. In this case $\|\cdot\|_{\delta}\le r\|\cdot\|_{L}$ on $F^{\delta}$, where $r\in\R$ is the smallest such that $F$ is $r+$-monotonically bounded.
\end{proposition}
\begin{proof}
(i): $F^{L}$ is an order complete vector lattice which contains $F$ as an order dense sublattice. Then, the ideal generated by $F$ in $F^{L}$ is an order complete vector lattice which contains $F$ as an order dense majorizing sublattice. Hence, this ideal is isomorphic to $F^{\delta}$.

For $g\in F^{\delta}_{+}$ and $t>\|g\|_{\delta}$ find $f\in F$ such that $f\ge g$ and $\|f\|\le t$. Then, $\|g\|_{L}\le \|f\|_{L}\le \|f\|\le t$, and since $t$ was arbitrary, we get $\|g\|_{L}\le \|g\|_{\delta}$.\medskip

(ii): Sufficiency: Let $P\subset\Ba_{F_{+}}$ be directed and dominable. Then, there is $g\in F^{u}$ such that $P\uparrow g$. It follows that $g\in \Ba_{F}^{L}\subset F^{\delta}$, and so there is $f\in F$ with $f\ge g$. We conclude that $P\le f$, hence $P$ is order bounded in $F$. Theorem \ref{lev} now guarantees that $F$ is monotonically bounded.\medskip

Necessity: First, by Theorem \ref{lev} $F$ is $s$-monotonically bounded, for some $s\ge 1$. Let $g\in F^{L}_{+}$ and let $t>\|g\|_{L}$. There is a directed $P\subset t\Ba_{F_{+}}$ such that $P\uparrow g$. Then, by $s$-monotone boundedness there is an upper bound for $P$ in $ts\Ba_{F}$, say $f$. Hence, $g\le f$, and since $g$ was arbitrary, it follows that $F$ is majorizing in $F^{L}$, and so by (i) we get $F^{L}=F^{\delta}$. Moreover, $\|g\|_{\delta}\le \|f\|\le st$, and since $t$ was arbitrary, we conclude that $\|g\|_{\delta}\le s\|g\|_{L}$. Taking infimum over all admissible $s$'s yields $\|\cdot\|_{\delta}\le r\|\cdot\|_{L}$.
\end{proof}

\begin{remark}
We say that $F$ has the \emph{$r$-Nakano property}\footnote{This terminology differs slightly from the one introduced in \cite{wickstead}, but it is in line with the other terms considered in the article.} if every directed order bounded subset of $\Ba_{F_{+}}$ has an upper bound of norm at most $r$, and the \emph{$r+$-Nakano property} if it has the \emph{$s$-Nakano property}, for every $s>r$. Clearly, $r$-, or $r+$-monotone boundedness implies the $r$-, or $r+$-Nakano property, respectively. The $r+$-Nakano property is equivalent to $r$-demicontinuity of $\left(F^{\delta},\|\cdot\|_{\delta}\right)$, according to \cite[Theorem 3]{wickstead}.
\qed\end{remark}

\section{Semicontinuity-like properties in AM-spaces}

Recall that by Kakutani theorem a Banach lattice $F$ is an AM-space if and only if it isometrically embeds as a sublattice of $C\left(K\right)$, for some compact Hausdorff space $K$ (see \cite[Theorem 3.40]{ab0}). It is also easy to see that $F$ embeds as an ideal of $C\left(K\right)$, for some compact Hausdorff space $K$ if and only if it is of the form $C_{0}\left(X\right)$ for some Hausdorff locally compact space $X$ (see also \cite[\S 9, Theorem 13]{lacey} and \cite[Theorem 4]{wickstead} for intrinsic characterizations of these normed lattices). Let us consider some related properties.

\begin{proposition}\label{bm}
If $F$ is a pointwise demicontinuous AM-space, then $F^{L}$ is isometrically isomorphic to $C\left(K\right)$, for some compact Hausdorff extremally disconnected space $K$.
\end{proposition}
\begin{proof}
Since $\Ba_{F_{+}}$ is directed, according to Proposition \ref{barely} it is dominable, i.e. order bounded in $F^{u}$. As the latter is order complete, there is $e\in F^{u}$ such that $\Ba_{F_{+}}\uparrow e$. Let $E$ be the ideal in $F^{u}$ generated by $e$. Then, $E$ is an order complete vector lattice with a strong unit, and so there is a lattice isomorphism from $E$ onto $C\left(K\right)$, for some compact Hausdorff extremally disconnected space $K$, which maps $e$ into $\1$ (combine \cite[Theorem 45.4]{lz} with \cite[Theorem 1.50]{ab0}). Note that this isomorphism is an isometry from the order unit norm $\|\cdot\|_{e}$ on $E$ into the uniform norm on $\Co\left(K\right)$.


It remains to show that $\Ba_{F}^{L}=\left[-e,e\right]$. Clearly, the latter is a solid order closed subset of $F^{u}$, which contains $\Ba_{F}$, and so $\Ba_{F}^{L}\subset \left[-e,e\right]$. Conversely, since $\Ba_{F_{+}}\uparrow e$, it follows that $e\in \Ba_{F}^{L}$, and since the latter set is solid, we conclude that $\left[-e,e\right]\subset\Ba_{F}^{L}$.
\end{proof}

Note that in the definitions of semicontinuity-related properties, in the case when $F$ is an AM-space, directed subsets of $\Ba_{F_{+}}$ may be replaced with arbitrary subsets. Indeed, if $P\subset\Ba_{F_{+}}$, then $P^{\vee}$ is a directed subset of $\Ba_{F_{+}}$ with the same supremum as $P$. In the following result part (ii) improves \cite[Theorem 6]{wickstead}. We will present another proof of part (i) in Theorem \ref{david}.

\begin{theorem}\label{main}
Let $F$ be an AM-space. Then:
\item[(i)] $F$ is semicontinuous if and only if it isometrically embeds as a regular sublattice of $C\left(K\right)$, for some compact Hausdorff space $K$.
\item[(ii)] $F$ is demicontinuous if and only if it isomorphically embeds as a regular sublattice of $C\left(K\right)$, for some compact Hausdorff space $K$. In this case it is topologically semicontinuous.
\item[(iii)] $F$ is pointwise demicontinuous if and only if it injectively embeds as a regular sublattice of $C\left(K\right)$, for some compact Hausdorff space $K$.\medskip

    Moreover, ``regular'' can be replaced with ``order dense'', while $K$ can be assumed extremally disconnected.
\end{theorem}
\begin{proof}
Sufficiency in (i) follows from the fact that the semicontinuity passes down to regular sublattices, while in (ii) one should also recall that a renorming of a semicontinuous norm is topologically semicontinuous, hence demicontinuous. For sufficiency in (iii) observe that a homomorphism between Banach lattices is norm continuous (see \cite[Theorem 5.19]{ab0}), and so the claim follows from Proposition \ref{barely2}.

Necessity in (iii) follows immediately from Proposition \ref{bm}, while necessity in (i) and (ii) follows from Proposition \ref{bm} and the fact that if $F$ is (weakly) semicontinuous, then the norm of $F$ is equal (equivalent) to the Lorentz seminorm.
\end{proof}

The following result unifies \cite[Proposition 4]{aat} and \cite[Proposition 2.1]{abt}.

\begin{proposition}
$F$ is $r$-monotonically bounded if and only if there is a strong unit $e$ such that $\|\cdot\|_{e}\le \|\cdot\|\le r\|\cdot\|_{e}$. $F$ is monotonically complete if and only if it is isomorphic to a space $\Co\left(K\right)$, for some compact Hausdorff extremally disconnected $K$.
\end{proposition}
\begin{proof}
First, note that $\|\cdot\|_{e}\le \|\cdot\|\le r\|\cdot\|_{e}$ is equivalent to $\Ba_{F}\subset\left[-e,e\right]\subset r\Ba_{F}$. These conditions clearly imply that $e$ is a strong unit.

If $F$ is $r$-monotonically bounded, then there is an upper bound $e$ for $\Ba_{F_{+}}$ of norm at most $r$, so that $e\in r\Ba_{F}$. Then, $\Ba_{F}\subset\left[-e,e\right]\subset r\Ba_{F}$ follows from solidity.

Conversely, if $e\in F_{+}$ is such that $\Ba_{F}\subset\left[-e,e\right]\subset r\Ba_{F}$, then $e$ serves as an upper bound for any directed $P\subset\Ba_{F}$, and $\|e\|\le r$.\medskip

It is clear that if $F$ is isomorphic to $\Co\left(K\right)$, for some compact Hausdorff extremally disconnected $K$, then $F$ is monotonically complete.

Conversely, suppose that $F$ is monotonically complete. Since then $F$ is monotonically bounded, hence has a strong unit, it is isomorphic to $\Co\left(K\right)$, for some compact Hausdorff space $K$. Since also $F$ is order complete, it follows that $K$ is extremally disconnected.
\end{proof}

We will now construct another representation of a semicontinuous AM-space as a regular sublattice of a $\Co\left(K\right)$-space. First, recall that if $B\subset E$ is a \emph{positively balanced} (i.e. such that $\left[0,1\right]B\subset B$)  set in a vector space, then $f:B\to\R$ is \emph{positively homogeneous} if $f\left(te\right)=tf\left(e\right)$, for every $e\in B$ and $0\le t\le 1$. Note that in this case if $e\in B$ and $0\le t,s\le 1$ then $sf\left(te\right)=tf\left(se\right)$.

Let $K_{F}$ be the set of vector lattice homomorphisms from $F$ into $\R$ which are of norm at most $1$, endowed with the weak* topology. It is easy to see that $K_{F}$ is positively balanced and compact. Recall that $F$ is an AM-space if and only if it is isomorphic to $\Co_{\ph}\left(K_{F}\right)$, the space of weak* continuous positively homogeneous functions on $K_{F}$ (see e.g. \cite[Proposition 5.5]{bgdmt}). Therefore, we will view $F$ as a sublattice of $\Co\left(K_{F}\right)$. Let us present an auxiliary Urysohn-style result.

\begin{proposition}\label{urysohn}
Let $x\in K_{F}\cap \So_{F^{*}}$ and let $V\subset K_{F}$ be an open weak* neighborhood of $x$. Then, there is $f\in \So_{F_{+}}$ and $0\le r<1$ such that $f\left(x\right)=1$ and $\left.f\right|_{K_{F}\backslash V}\le r\1$.
\end{proposition}
\begin{proof}
Let $Y:=K_{F}\backslash V$, which is weak* compact. We first show that for every $y\in Y$ there is $f_{y}\in \So_{F_{+}}$ such that $f_{y}\left(x\right)=1$ and $f_{y}\left(y\right)<1$.

Assume $y=tx$, for some $t\ge 0$. It follows from $y\notin V\ni x$ that $y\ne x$, but since $\|x\|=1$, we conclude that $t<1$. By \cite[Theorem 4.3]{bgdmt} there is $f_{y}\in \So_{F_{+}}$ such that $f_{y}\left(x\right)=1$, and then $f_{y}\left(y\right)=t<1$.

We can now assume that $x,y$ are linearly independent, and so by \cite[Corollary 4.6]{bgdmt} there is $f_{y}\in \So_{F_{+}}$ such that $f_{y}\left(x\right)=1$ and $f_{y}\left(\frac{1}{\|y\|}y\right)=0$. It then follows that $f_{y}\left(y\right)=0<1$.\medskip

For every $y\in Y$ let $U_{y}:=\left\{z\in K_{F},~ f_{y}\left(z\right)<\frac{1+f_{y}\left(y\right)}{2}\right\}$, which is a weak* open neighborhood of $y$. By compactness there are $y_{1},...,y_{n}\in Y$ such that $Y\subset U_{y_{1}}\cup...\cup U_{y_{n}}$. Let $f:=f_{y_{1}}\wedge...\wedge f_{y_{n}}\in\Ba_{F_{+}}$ and let $r:=\frac{1+f_{y_{1}}\left(y_{1}\right)\vee...\vee f_{y_{n}}\left(y_{n}\right)}{2}<1$. We have that $f\left(x\right)=f_{y_{1}}\left(x\right)\wedge...\wedge f_{y_{n}}\left(x\right)=1$, hence $f\in \So_{F_{+}}$. At the same time, if $y\in Y$, then $y\in U_{y_{k}}$, for some $k=1,...,n$, and so $f\left(y\right)\le f_{y_{k}}\left(y\right)<\frac{1+f_{y_{k}}\left(y_{k}\right)}{2}\le r$.
\end{proof}

The following theorem and remark answer a question by David de Hevia.

\begin{theorem}\label{david}
$F$ is semicontinuous if and only if it is regular in $\Co\left(K_{F}\right)$.
\end{theorem}
\begin{proof}
Sufficiency follows from the fact that the semicontinuity passes down to regular sublattices. Therefore, we only need to show necessity. Let $P\subset F_{+}$ be such that $P\downarrow_{F} \0$. Assume that $\0< h\in\Co\left(K_{F}\right)$ is a lower bound for $P$. There is $x\in K_{F}$ such that $h\left(x\right)>0$. Let $h':K_{F}\to \R_{+}$ be defined by $h'\left(y\right):=\frac{1}{\|x\|}h\left(\|x\|y\right)$. Clearly, $h'$ is weak* continuous, and for every $p\in P$ and $y\in K_{F}$ positive homogeneity of $p$ yields $h'\left(y\right)=\frac{1}{\|x\|}h\left(\|x\|y\right)\le \frac{1}{\|x\|}p\left(\|x\|y\right)=p\left(y\right)$. Hence, we can replace $h$ with $h'$ and $x$ with $\frac{1}{\|x\|}x$ if needed and assume that $\|x\|=1$. Also, by rescaling both $P$ and $h$, we may assume that $h\left(x\right)>1$.\medskip

Let $V:=\left\{y\in K_{F},~ h\left(y\right)>1\right\}$, which is a weak* open neighborhood of $x$. Proposition \ref{urysohn} yields $f\in \So_{F_{+}}$ and $0\le r<1$ such that $\left.f\right|_{K_{F}\backslash V}\le r\1$. It follows from $P\downarrow_{F} \0$ that $\left(f-P\right)^{+}\uparrow_{F} f$. For every $p\in P$ and $y\in V$ we have that $f\left(y\right)\le\|f\|=1<h\left(y\right)\le p\left(y\right)$, and so $\left(f-p\right)^{+}\left(y\right)\le \left(f\left(y\right)-h\left(y\right)\right)^{+}=0$. If $p\in P$ and $y\notin V$ we have $\left(f-p\right)^{+}\left(y\right)\le f\left(y\right)\le r$. We conclude that $\left(f-P\right)^{+}\subset r\Ba_{F_{+}}$, and so the semicontinuity gives $\|f\|\le r<1$, contradiction.
\end{proof}

\begin{remark}
We can also view $\Co_{\ph}\left(K_{F}\right)$ embedded into $\Co\left(L_{F}\right)$ by restriction, where $L_{F}:=K_{F}\cap\So_{F^{*}}$. Note that $L_{F}$ may not be compact, or even locally compact (see \cite[\S 9, Theorem 13]{lacey}). We claim that if $F$ is semicontinuous, then $\Co_{\ph}\left(K_{F}\right)$ embeds as a regular sublattice of $\Co\left(L_{F}\right)$. If $P\subset F$ is such that $P\downarrow_{F}\0$ but $h\in \Co\left(L_{F}\right)_{+}$ is a lower bound for $P$ and $x\in L_{F}$ is such that $h\left(x\right)>1$, then $U:=\left\{y\in L_{F},~ h\left(y\right)>1\right\}$, is a weak* open neighborhood of $x$ in $L_{F}$, and so there is a weak* open set $V\subset K_{F}$ such that $V\cap L_{F}=U$. After this, the proof proceeds as in Theorem \ref{david}. It also follows that $F$ embeds as a regular sublattice of $\Co\left(\overline{L_{F}}\right)$.
\qed\end{remark}

In order to consider another representation of AM-spaces we need to recall a concept from the Banach space theory. Let $E$ be a normed space. Recall that the \emph{bounded weak* topology} on $E^{*}$ is the strongest topology which agrees with the weak* topology on the norm bounded subsets of $E^{*}$. This topology is the topological modification of the \emph{bounded weak* convergence} on $E^{*}$, where a net is convergent if it weak* converges (to the same limit) and has a norm bounded tail. It is easy to see that $f:E^{*}\to \R$ is continuous with respect to bounded weak* topology or convergence if and only if it is weak* continuous on norm bounded subsets of $E^{*}$. Finally, a crucial property of the bounded weak* topology is that it agrees with the topology of uniform convergence on compact sets in $E$ on $\Co\left(E\right)$ restricted to $E^{*}$. It is determined by the seminorms $\rho_{K}$, given by $\rho_{K}\left(x\right):=\bigvee\limits_{e\in K}\left|e\left(x\right)\right|$, where $K\subset E$ is compact. In particular, $\rho_{K}$ is a bounded weak* continuous function on $E^{*}$, and also it is the pointwise supremum of $K\cup -K$ (here we view $E$ embedded into $\Co\left(E^{*},\mathrm{bw}^{*}\right)$). More information about bounded weak* topology can be found in \cite[Section 8.8]{at} and \cite[Section 8]{bctv}.\medskip

Let $F$ be an AM-space. As mentioned before it is isometrically isomorphic to $\Co_{\ph}\left(K_{F}\right)$. Let $X_{F}$ be the set of all homomorphisms from $F$ to $\R$, endowed with the bounded weak* topology. Clearly, $X_{F}=\bigcup\limits_{n\in\N}K_{F}$.

\begin{proposition}\label{bouw}
$F$ is isomorphic to $\Co_{\ph}\left(X_{F}\right)$.
\end{proposition}
\begin{proof}
Let $R:\Co_{\ph}\left(X_{F}\right)\to \Co_{\ph}\left(K_{F}\right)$ be the restriction operator. It is clear that it is a homomorphism and its injectivity follows from positive homogeneity of elements of $\Co_{\ph}\left(X_{F}\right)$. It is left to prove that $R$ is surjective. Let $g\in \Co_{\ph}\left(K_{F}\right)$ and define $\hat{g}:X_{F}\to\R$ by $\hat{g}\left(0_{F^{*}}\right)=0$ and $\hat{g}\left(x\right):=\|x\|g\left(\frac{1}{\|x\|}x\right)$, when $x\ne 0_{F^{*}}$. Note that if $r\ge \|x\|$, then $\hat{g}\left(x\right)=rg\left(\frac{1}{r}x\right)$. Clearly, $\hat{g}$ is positively homogeneous. To prove that $\hat{g}\in \Co_{\ph}\left(X_{F},\mathrm{bw}^{*}\right)$ it is enough to show that $\left.\hat{g}\right|_{rK_{F}}\in \Co_{\ph}\left(rK_{F}\right)$, for every $r>0$. However, $\hat{g}\left(x\right)=r g\left(\frac{1}{r}x\right)$, for every $x\in rK_{F}$, and so $\left.\hat{g}\right|_{rK_{F}}$ is weak* continuous. Since $g=R\hat{g}$, we have established surjectivity of $R$.
\end{proof}

In the sequel we will identify $F$ with $\Co_{\ph}\left(X_{F}\right)$. Note that $\left.\rho_{L}\right|_{X_{F}}\in \Co_{\ph}\left(X_{F}\right)$, for every compact $L\subset F$. We also have the following  Urysohn-style result inspired by \cite[Theorem 7]{ad}.

\begin{proposition}\label{yussef}
If $0_{F^{*}}\ne x\in X_{F}$, and $U$ is a bounded weak* open neighborhood of $x$, there is $f\in F$ which vanishes outside of $W:=\left(0,+\8\right)U$ with $f\left(x\right)>0$ and such that $f\le g$, for every $g\in F$ such that $g\ge \1_{U}$.
\end{proposition}
\begin{proof}
Since the bounded weak* topology agrees with the topology of uniform convergence on compacts in $F$, there is a compact set $L\subset F$ such that $V:=\left\{y\in X_{F},~\rho_{L}\left(y-x\right)<1\right\}\subset U$. Note that $\left.\rho_{L}\right|_{X_{F}}\in F$. If $\rho_{L}\left(x\right)=0$, find $e\in F$ such that $e\left(x\right)=1$, and replace $L$ with $L\cup\left\{e\right\}$; then $\rho_{L}\left(x\right)=1$. Hence, we may assume that $t:=\rho_{L}\left(x\right)>0$.

Let $h:X_{F}\to\R_{+}$ be defined by $h\left(y\right):=\rho_{L}\left(ty-\rho_{L}\left(y\right)x\right)$. Clearly, $h$ is positively homogeneous; its bounded weak* continuity follows from that of $\rho_{L}$. Also, note that $h\left(x\right)=0$.

Consider $f\in \Co_{\ph}\left(X_{F}\right)$ defined by $f:=\frac{1}{t+1}\left(\left.\rho_{L}\right|_{X_{F}}-h\right)^{+}$. It follows that $f\left(x\right)=\frac{t}{t+1}>0$. If $\rho_{L}\left(y\right)= 0$, then $f\left(y\right)=0$. Assume that $\rho_{L}\left(y\right)>0$ and $y\notin \left(0,+\8\right)V$. Then, for every $r> 0$ we have $\frac{1}{r}y\notin V$, and so $\rho_{L}\left(\frac{1}{r}y-x\right)\ge 1$. Take $r:=\frac{\rho_{L}\left(y\right)}{t}>0$. We have $\rho_{L}\left(\frac{t}{\rho_{L}\left(y\right)}y-x\right)\ge 1$, therefore $h\left(y\right)=\rho_{L}\left(ty-\rho_{L}\left(y\right)x\right)\ge \rho_{L}\left(y\right)$, and so $f\left(y\right)=0$. Thus, $f$ vanishes outside of $W$.\medskip

Now assume that $g\in F$ is such that $g\ge\1_{U}$. For every $y\notin \left(0,+\8\right)V$ we have that $f\left(y\right)=0\le g\left(y\right)$. Otherwise, $y=rz$, where $r>0$ and $z\in V\subset U$, so that $\rho_{L}\left(z-x\right)<1$, and so $\rho_{L}\left(z\right)\le \rho_{L}\left(x\right)+1=t+1$. Thus, $f\left(z\right)\le \frac{1}{t+1}\rho_{L}\left(z\right)\le 1\le g\left(z\right)$, which implies $f\left(y\right)=rf\left(z\right)\le rg\left(z\right)=g\left(y\right)$. Combining the two cases we conclude that $f\le g$.
\end{proof}

\begin{theorem}\label{regu}
$F$ is a regular sublattice of $\Co\left(X_{F}\right)$.
\end{theorem}
\begin{proof}
Assume that $P\subset F$ is such that $\bigwedge_{F}P=\0$. If $h\in \Co\left(X_{F}\right)$ is such that $P\ge h>\0$, by scaling we may assume that $P\ge\1_{U}$, where $U\subset X_{F}$ is bounded weak* open. By Proposition \ref{yussef} there is $f\in F$ such that $P\ge f>\0$, which contradicts the assumption about infimum. We conclude that $\0$ is the only lower bound for $P$ in $\Co\left(X_{F}\right)_{+}$, and so $\bigwedge_{\Co\left(X_{F}\right)}P=\0$.
\end{proof}

\section{Semicontinuity-like properties in free AM-spaces}\label{s8}

We now present a broad class of anti-demicontinuous AM-spaces. Let $E$ be a normed space. The set $\Co_{\ph}\left(\Ba_{E^{*}},\mathrm{w}^{*}\right)$ of positively homogeneous weak* continuous functions on $\Ba_{E^{*}}$ is a closed sublattice of $\Co\left(\Ba_{E^{*}},\mathrm{w}^{*}\right)$, and so it is an AM-space.

We view $E$ as isometrically embedded into $\Co_{\ph}\left(\Ba_{E^{*}},\mathrm{w}^{*}\right)$ via $e\left(x\right):=\left<x,e\right>$, for $e\in E$ and $x\in \Ba_{E^{*}}$. In fact, $E$ generates $\Co_{\ph}\left(\Ba_{E^{*}},\mathrm{w}^{*}\right)$ as a closed sublattice. Indeed, the closed sublattice of $\Co\left(\Ba_{E^{*}},\mathrm{w}^{*}\right)$ generated by $E$ is the set of all weak* continuous functions $f$ on $\Ba_{E^{*}}$ which satisfy the same \emph{constraints} as elements of $E$, i.e. equations of the form $tf\left(x\right)=sf\left(y\right)$, where $t,s\ge 0$ and $x,y\in \Ba_{E^{*}}$ (see e.g. \cite[Theorem 2.1]{bt}), and it is easy to show using Hahn--Banach theorem that the only constraints satisfied by elements of $E$ are of the form $tf\left(sx\right)=sf\left(tx\right)$, where $0\le t,s\le 1$ and $x\in \Ba_{E^{*}}$.

Spaces of this type have attracted attention recently in relation to free Banach lattices, and in particular, $\Co_{\ph}\left(\Ba_{E^{*}},\mathrm{w}^{*}\right)$ is the \emph{free AM-space} over $E$, in the sense that whenever $F$ is an AM-space and $T:E\to F$ is a linear contraction, there is a contractive homomorphism $\widehat{T}:\Co_{\ph}\left(\Ba_{E^{*}},\mathrm{w}^{*}\right)\to F$ which extends $T$ (see \cite[Proposition 2.2]{ottt}).

\begin{theorem}\label{cph}
For a normed space $E$ the following conditions are equivalent:
\item[(i)] $\dim E<\8$;
\item[(ii)] $F:=\Co_{\ph}\left(\Ba_{E^{*}},\mathrm{w}^{*}\right)$ is a regular sublattice of $\Co\left(\Ba_{E^{*}},\mathrm{w}^{*}\right)$;
\item[(iii)] $F$ is $1$-monotonically bounded;
\item[(iv)] $F$ is semicontinuous, that is $\Ba_{F}$ is order closed in $F$;
\item[(v)] $F$ is pointwise demicontinuous;
\item[(vi)] $F$ is not anti-demicontinuous, that is the order adherence of $\Ba_{F}$ is not equal to $F$.
\end{theorem}
\begin{proof}
(i)$\Rightarrow$(ii): The case $\dim E=0$ is trivial, and so we may assume that $\dim E\in\N$. Let $P\subset F$ be such that $\bigwedge_{F}P=\0$. Assume that $h\in \Co\left(\Ba_{E^{*}},\mathrm{w}^{*}\right)_{+}$ is such that $h\le P$. Since every element of $P$ vanishes at $0_{E^{*}}$, it follows that $h\left(0_{E^{*}}\right)=0$. Let $0<r\le 1$ and let $g:\Ba_{E^{*}}\to \R$ be defined by $g\left(tx\right):=\frac{t}{r}h\left(rx\right)$, where $0\le t\le 1$ and $x\in \So_{E^{*}}$. It is easy to see that $\dim E<\8$ yields continuity of $g$, hence $g\in F$, and for every $p\in P$, $0\le t\le 1$ and $x\in \So_{E^{*}}$ we have $g\left(tx\right)=\frac{t}{r}h\left(rx\right)\le \frac{t}{r}p\left(rx\right)=p\left(tx\right)$. Thus, $g\le P$, and $\bigwedge_{F}P=\0$ imply $g\le\0$, hence $h\left(rx\right)=rg\left(x\right)=0$, for every $x\in\So_{E^{*}}$. Since $r$ was arbitrary, we conclude that $h=\0$. Thus, $\bigwedge_{\Co\left(\Ba_{E^{*}},\mathrm{w}^{*}\right)}P=\0$.\medskip

(i)$\Rightarrow$(iii) follows from the fact that if $\dim E=n\in\N$, then $F\simeq\Co\left(S^{n-1}\right)$. Since we have established in Proposition \ref{lef} that $1$-monotone boundedness implies semicontinuity, we get (iii)$\Rightarrow$(iv). (ii)$\Rightarrow$(iv) follows from the fact that semicontinuity passes down to regular sublattices. (iv)$\Rightarrow$(v)$\Rightarrow$(vi) is trivial.\medskip

(vi)$\Rightarrow$(i): We assume that $\dim E=\8$ and prove that $\|\cdot\|_{L}\equiv 0$. Fix $0_{E}\ne e \in E$ and $\delta>0$, and let $Q:=\left\{f\in F,~ f\ge\left(e-\delta\1\right)^{+}\right\}$. Clearly, $Q$ is closed with respect to infimum of finite subsets, hence it is downward directed. We claim that $\bigwedge_{F}Q=\0$.

Assume that $h\in F_{+}$ is such that $\0< h\le Q$. There is $y\in \Ba_{E^{*}}$ with $h\left(y\right)>0$. Note that $e^{+}\in Q$, hence $0<h\left(y\right)\le e^{+}\left(y\right)$, and so $e\left(y\right)>0$.

Let $G$ be the annihilator of $e$ in $E^{*}$, which is weak* closed in $E^{*}$. It follows that $\Ba_{G}$ is weak* compact in $E^{*}$, while since $\dim G=\8$, its unit ball is not norm compact. We conclude that the weak* topology of $E^{*}$ does not agree with the norm topology on $G$. It is now easy to find a net $\left(x_{p}\right)_{p\in P}\subset\So_{G}$ which is weak*-null in $E^{*}$.

Fix $r:=\frac{\delta}{4\left(\|e\|+\delta\right)}\le\frac{1}{2}$ and note that $\Ba_{E^{*}}\ni ry + \frac{1}{2}x_{p}\toww ry$. Since $h$ is weak*-continuous and $h\left(ry\right)=rh\left(y\right)>0$, it follows that there is $p\in P$ such that $h\left(ry+ \frac{1}{2}x\right)>0$, where $x:=x_{p}\in \So_{G}$. Let $g\in \So_{E}$ be such that $g\left(x\right)\ge\frac{1}{2}$. As $\|y\|,\|g\|\le1$, we get $$\frac{1}{r}g\left(ry+ \frac{1}{2}x\right)=\frac{1}{2r}g\left(x\right)+g\left(y\right)\ge \frac{1}{4r}-1=\frac{\left(\|e\|+\delta\right)}{\delta}-1\ge \frac{e\left(y\right)}{\delta},$$
hence $\delta g\left(ry+ \frac{1}{2}x\right)\ge re\left(y\right)=e\left(ry+ \frac{1}{2}x\right)$, where the equality follows from $x\in G$.

We have that $\left(e-\delta g\right)^{+}\in F$ and it follows from $\left.g\right|_{\Ba_{E^{*}}}\le\1$ that $\left(e-\delta g\right)^{+}\ge\left(e-\delta\1\right)^{+}$. Hence, $\left(e-\delta g\right)^{+}\in Q$, but $$\left(e-\delta g\right)^{+}\left(ry+ \frac{1}{2}x\right)=\left(e\left(ry+ \frac{1}{2}x\right)-\delta g\left(ry+ \frac{1}{2}x\right)\right)^{+}=0< h\left(ry+ \frac{1}{2}x\right),$$ contradiction. This concludes the proof of the claim.

It follows that $\left(e-Q\right)^{+}\uparrow e^{+}$, and at the same time for every $q\in Q$ we have $q\ge \left(e-\delta \1\right)^{+}=e-e\wedge\delta\1$, and so $\0\le \left(e-q\right)^{+}\le \left(e-e+e\wedge\delta\1\right)^{+}\le \delta\1$, implying that $\left(e-q\right)^{+}\in \delta\Ba_{F_{+}}$. Since $\delta$ was arbitrary we conclude that $\|e^{+}\|_{L}=0$. Applying the same argument to $-e$ yields $\|e^{-}\|_{L}=0$, hence $\|e\|_{L}=0$. As $e$ was arbitrary, it follows that $\ker\|\cdot\|_{L}$ is a closed ideal in $F$ which contains $E$. The discussion before the theorem implies that such an ideal must be equal to $F$, and so $\|\cdot\|_{L}\equiv 0$.
\end{proof}

\begin{remark}\label{noidea}
If $\dim E>0$ then $\Co_{\ph}\left(\Ba_{E^{*}},\mathrm{w}^{*}\right)$ is neither an ideal, nor an order dense sublattice of $\Co\left(\Ba_{E^{*}},\mathrm{w}^{*}\right)$. Indeed, in order for that to happen $\Co_{\ph}\left(\Ba_{E^{*}},\mathrm{w}^{*}\right)$ has to be regular in $\Co\left(\Ba_{E^{*}},\mathrm{w}^{*}\right)$, hence by Theorem \ref{cph} we get $n:=\dim E\in\N$ and $\Co_{\ph}\left(\Ba_{E^{*}},\mathrm{w}^{*}\right)\simeq\Co\left(S^{n-1}\right)$, which is neither order dense nor an ideal in $\Co\left(\Ba_{\R^{n}}\right)$.
\qed\end{remark}

The close connection of $\Co_{\ph}\left(\Ba_{E^{*}},\mathrm{w}^{*}\right)$ with the free Banach lattice over $E$ inspires the following question (partial results in this direction are \cite[Theorems 3.4 and 3.6]{abt}).

\begin{question}
Characterize those normed spaces $E$ such that the free Banach lattices over $E$ have the properties considered in this article.
\end{question}

We now consider another representation of $F:=\Co_{\ph}\left(\Ba_{E^{*}},\mathrm{w}^{*}\right)$. Let $\delta:E^{*}\to F^{*}$ be given by $\delta_{0_{E^{*}}}\left(f\right):=0$ and $\delta_{x}\left(f\right):=\|x\|f\left(\frac{1}{\|x\|}x\right)$, for $x\ne 0_{E^{*}}$. Note that if $r\ge \|x\|$, then $\delta_{x}\left(f\right)=rf\left(\frac{1}{r}x\right)$. This is a positively homogeneous map from $E^{*}$ into $X_{F}$, and it also preserves norm. In fact, more is true.

\begin{proposition}\label{pho}
$\delta$ is a homeomorphism from $\left(E^{*},\mathrm{bw}^{*}\right)$ onto $X_{F}$.
\end{proposition}
\begin{proof}
Since $F$ is a closed sublattice of $\Co\left(\Ba_{E^{*}},\mathrm{w}^{*}\right)$, by \cite[Proposition 5.4]{bgdmt} all homomorphisms on $F$ are of the form $r\delta_{x}$, where $x\in \Ba_{E^{*}}$. It is easy to deduce from this that $\delta$ is a surjection. Assume that $\left(x_{p}\right)_{p\in P}\subset E^{*}$ is bounded weak* convergent to $x\in E^{*}$. Then by passing to a tail we may assume that there is $r>0$ such that $x\in r\Ba_{E^{*}}$ and $x_{p}\in r\Ba_{E^{*}}$, for every $p\in P$. Then, $\delta_{x}\in rK_{F}$ and $\delta_{x_{p}}\in rK_{F}$, for every $p\in P$. Since $x_{p}\toww x$, for every $f\in F$ we have $\delta_{x_{p}}\left(f\right)=rf\left(\frac{1}{r}x_{p}\right)\to rf\left(\frac{1}{r}x\right)=\delta_{x}\left(f\right)$. We conclude that $\delta_{x_{p}}\tobw \delta_{x}$. Conversely assume that $\delta_{x_{p}}\tobw \delta_{x}$. Using positive homogeneity again we may assume that there is $r>0$ such that $x\in r\Ba_{E^{*}}$ and $x_{p}\in r\Ba_{E^{*}}$, for every $p\in P$. For every $e\in E$ viewed as an element of $F$ we have $x_{p}\left(e\right)=\delta_{x_{p}}\left(e\right)\to \delta_{x}\left(e\right)=x\left(e\right)$. Hence, $x_{p}\tobw x$.
\end{proof}

Combining Proposition \ref{pho} with Proposition \ref{bouw} and Theorem \ref{regu} we get the following consequences.

\begin{corollary}\label{regu2}
$F$ is isomorphic to $\Co_{\ph}\left(E^{*},\mathrm{bw}^{*}\right)$, and it is regular in $\Co\left(E^{*},\mathrm{bw}^{*}\right)$.
\end{corollary}

\begin{remark}
Analogously to Remark \ref{noidea} let us point out that $\Co_{\ph}\left(E^{*},\mathrm{bw}^{*}\right)$ is neither order dense, nor an ideal in $\Co\left(E^{*},\mathrm{bw}^{*}\right)$. If $U\ne E^{*}$ is a closed neighborhood of $0_{E^{*}}$, there is $\0<g\in\Co\left(E^{*},\mathrm{bw}^{*}\right)$ which vanishes on $U$. Then, any positively homogeneous function $f$ satisfying $\0\le f\le g$ must be equal to $\0$, and so $\Co_{\ph}\left(E^{*},\mathrm{bw}^{*}\right)$ is not order dense. On the other hand, take $x\ne 0_{E^{*}}$ and find $f\in \Co_{\ph}\left(E^{*},\mathrm{bw}^{*}\right)_{+}$ such that $f\left(x\right)=1$. Find $g\in \Co\left(E^{*},\mathrm{bw}^{*}\right)_{+}$ such that $g\left(x\right)=1$ and $g\left(2x\right)=0$. Then, $g\wedge f\notin \Co_{\ph}\left(E^{*},\mathrm{bw}^{*}\right)_{+}$ is in the ideal generated by $\Co_{\ph}\left(E^{*},\mathrm{bw}^{*}\right)$ in $\Co\left(E^{*},\mathrm{bw}^{*}\right)$, hence the former is not an ideal.
\qed\end{remark}

It now follows that uo convergence on $\Co_{\ph}\left(E^{*},\mathrm{bw}^{*}\right)$ is the restriction of uo convergence on $\Co\left(E^{*},\mathrm{bw}^{*}\right)$, which can be described explicitly (see e.g. \cite[Theorem 7.1]{erz}), while order convergence can be described using the aforementioned description of uo convergence together with the fact that a net is order null iff it is uo null and eventually order bounded. However, it might be possible to refine these criteria to better use the additional structure of $\Co_{\ph}\left(E^{*},\mathrm{bw}^{*}\right)$.

\begin{question}
Find the most efficient criteria for uo and order convergence in $\Co_{\ph}\left(E^{*},\mathrm{bw}^{*}\right)$, as well as descriptions of order bounded and dominable sets.
\end{question}

Since $\Ba_{E^{*}}$ is bounded weak* compact, the topology of uniform convergence on bounded weak* compact sets in $\Co\left(E^{*},\mathrm{bw}^{*}\right)$ restricted to $E$ agrees with the norm topology.

\begin{proposition}
If $E$ is a Banach space, then for $f\in \Co_{\ph}\left(E^{*},\mathrm{bw}^{*}\right)$ the set $K:=\left\{e\in E,~ e\le f\right\}$ is compact in $E$.
\end{proposition}
\begin{proof}
First, note that for $f^{*}\in \Co_{\ph}\left(E^{*},\mathrm{bw}^{*}\right)$ defined by $f^{*}\left(x\right):=-f\left(-x\right)$ and $e\in E$ we have that $e\le f$ iff $e\ge f^{*}$. Hence, $K=\left[f^{*},f\right]_{E}$. By replacing $f$ with $\left|f\right|\vee\left|f^{*}\right|$ if needed we may assume that $f\ge\0$ is even, and so $K=\left[-f,f\right]_{E}$. It follows that $K$ is closed and bounded in $E$, and so it is enough to show that it is equicontinuous. For $\varepsilon>0$ let $U:=f^{-1}\left[0,\varepsilon\right)$, which is a bounded weak*-open neighborhood of $0_{E^{*}}$. If $x\in E^{*}$, $y\in x+U$ and $e\in K$, we have that $\left|e\left(y\right)-e\left(x\right)\right|=\left|e\left(y-x\right)\right|\le f\left(y-x\right)\le\varepsilon$, as $y-x\in U$.
\end{proof}

We now briefly turn to the free AM-spaces over lattices. Let $M$ be a distributive lattice. Let $M^{*}$ be the set of all lattice homomorphisms from $M$ into $\left[-1,1\right]$, and let $M'$ be the set of all bounded lattice homomorphisms from $M$ into $\R$. We endow $M^{*}$ with the pointwise topology, turning it into a compact Hausdorff space. We also endow $M'$ with the bounded pointwise topology; clearly $M'=\bigcup\limits_{r>0}rM^{*}\subset\ell_{\8}\left(M\right)$. It was proven in \cite[Propositions 7.4 and 7.5]{bgdmt} that $F:=\Co_{\ph}\left(M^{*}\right)$ is the free AM-space over $M$, and its renorming is the free Banach lattice over $M$. As before we view $M\subset \Ba_{F}$, and in fact the sublattice generated by $M$ in $F$ is dense. Let $\delta: M^{*}\to K_{F}$ be defined by $\delta_{x}\left(f\right):=f\left(x\right)$. It is not hard to show that $\delta$ is a positively homogeneous homeomorphism, and it follows from there that $M'$ is positively homogeneously homeomorphic to $X_{F}$. The following now is a consequence of Theorems \ref{david} and \ref{regu}.

\begin{proposition}
$F$ is isomorphic to $\Co_{\ph}\left(M'\right)$ and it is always regular in $\Co\left(M'\right)$. It is regular in $\Co\left(M^{*}\right)$ if and only if it is semicontinuous.
\end{proposition}

Also, note that if $x\in M'$ and $f\in\Co_{\ph}\left(M'\right)\simeq \Co_{\ph}\left(M^{*}\right)$, then $\left|f\left(x\right)\right|\le\|f\|\|x\|$. Let us present some partial results concerning semicontinuity of $F$. Part (i) in the next result complements \cite[Corollary 3.3]{brt}.

\begin{proposition}\label{faml}
\item[(i)] If $M$ has the least and greatest elements, then $F$ is $1$-monotonically bounded.
\item[(ii)] If $M$ is totally ordered, but has neither least nor greatest elements, then $F$ is anti-demicontinuous.
\end{proposition}
\begin{proof}
(i): Assume that $M$ has the infimum $m$ and supremum $n$. Then for every $x\in M^{*}$ we have that $n\left(x\right)=\bigvee x\left(M\right)$ and $m\left(x\right)=\bigwedge x\left(M\right)$, therefore $\left|n\left(x\right)\right|\vee \left|m\left(x\right)\right|=\|x\|_{\8}$, and so $\left|n\right|\vee\left|m\right|=\|\cdot\|_{\8}$ on $M^{*}$. In particular, it follows that $\|\cdot\|_{\8}$ is continuous on $M^{*}$.

Let $Y:=\left\{x\in M^{*},~ \|x\|_{\8}=1\right\}$. Let $R:F\to \Co_{b}\left(Y\right)$ be the restriction map, which is clearly a contractive injective homomorphism. We claim that $R$ is a surjective isometry. Once we prove that, $1$-monotone boundedness of $F$ follows immediately, because $\Co_{b}\left(Y\right)$ has this property.

First, let us show that $R$ is an isometry. Let $f\in \So_{F_{+}}$. Due to compactness of $M^{*}$, there is $x\in M^{*}$ such that $f\left(x\right)=1$. Then, $x\ne 0_{M^{*}}$, and so $1=\|f\|\ge f\left(\frac{1}{\|x\|}x\right)=\frac{1}{\|x\|}f\left(x\right)=\frac{1}{\|x\|}\ge 1$ implies $x\in Y$. It follows that $\|Rf\|_{\8}=1$, as required.

To prove surjectivity let $g\in \Co_{b}\left(Y\right)$ and define $\hat{g}:M^{*}\to\R$ by $\hat{g}\left(0_{M^{*}}\right):=0$ and $\hat{g}\left(x\right):=\|x\|_{\8}g\left(\frac{1}{\|x\|_{\8}}x\right)$, for $x\ne 0_{M^{*}}$. Let us prove that $\hat{g}$ is continuous on $M^{*}$. Suppose that $x_{p}\to x$. which also implies $\|x_{p}\|_{\8}\to \|x\|_{\8}$. Assume first that $x=0_{L^{*}}$, so that $\|x_{p}\|_{\8}\to 0$, and since $\|g\|_{\8}<\8$, we have that $\hat{g}\left(x_{p}\right)\le \|g\|_{\8}\|x_{p}\|_{\8}\to 0=\hat{g}\left(x\right)$. If $x\ne 0_{M^{*}}$, by passing to a tail we may assume that $x_{p}\ne 0_{M^{*}}$, for every $p$. Then, $\frac{1}{\|x_{p}\|_{\8}}x_{p}\to \frac{1}{\|x\|_{\8}}x$, therefore $g\left(\frac{1}{\|x_{p}\|_{\8}}x_{p}\right)\to g\left(\frac{1}{\|x\|_{\8}}x\right)$, and so $\hat{g}\left(x_{p}\right)\to \hat{g}\left(x\right)$. Now $\hat{g}\in F$, and $R\hat{g}=g$, thus proving surjectivity of $R$.\medskip

(ii): We now assume that $M$ is totally ordered but has no greatest element, take $m\in M$ and show that $\|m^{+}\|_{L}=0$. To that end it is enough to prove that  for every $r>1$ we have that $\bigvee\limits_{n\in M}\left(n^{+}\wedge rm^{+}\right)= rm^{+}$.

Assume that $f\in F$ is such that $f\ge n^{+}\wedge rm^{+}$, for every $n\in M$. Let $x\in M^{*}$. For $n\in M$ define $y_{n}:M\to \left[-1,1\right]$ by $y_{n}\left(l\right):=\frac{1}{r}x\left(l\right)$, if $l< n$, and $y_{n}\left(l\right):=1$, otherwise. Since $M$ is totally ordered, it is easy to check that $y_{n}\in M^{*}$, for every $n\in M$ and that $y_{n}\to \frac{1}{r}x$ in $M^{*}$. For every $n>m$ we have that $f\left(y_{n}\right)\ge n\left(y_{n}\right)^{+}\wedge rm\left(y_{n}\right)^{+}=1\wedge r\frac{1}{r}m\left(x\right)^{+}=m\left(x\right)^{+}$. Continuity and positive homogeneity now yield $f\left(x\right)=rf\left(\frac{1}{r}x\right)=r\lim\limits_{n>m} f\left(y_{n}\right)\ge rm\left(x\right)^{+}$. It follows that $f\ge rm^{+}$, as required.

Analogously, one can show that if $M$ is totally ordered but has no least element, then $\|m^{-}\|_{L}=0$, for every $m\in M$. In the case when $M$ has neither extrema we get that $\|\cdot\|_{L}$ vanishes on $M$. Since the sublattice generated by $M$ in $F$ is dense,  anti-demicontinuity follows.
\end{proof}

The following example shows that it is possible for $\Co_{\ph}\left(M^{*}\right)$ to be semicontinuous even for unbounded $M$.

\begin{example}\label{finn}
Let $M$ be the set of all finite subsets of $\N$, including $\varnothing$. For $r,s\in\R$ and $n\in\N$ let $y^{r,s,n}$ be a sequence with $y^{r,s,n}_{n}=s$ and $y^{r,s,n}_{m}=r$, for $m\ne n$. Let $K:=\left\{y^{r,s,n},~ -1\le r\le s\le 1,~ n\in\N\right\}$, which is a compact positively balanced subset $\R^{\N}$. Let us show that $K$ is positively homogeneously homeomorphic to $M^{*}$.

Let $\varphi:M^{*}\to\R^{\N}$ be defined by $\varphi\left(x\right)=\left(x\left(\left\{n\right\}\right)\right)_{n\in\N}$. For $-1\le r\le s\le 1$ and $n\in\N$ we have $y^{r,s,n}=\varphi\left(x\right)$, where $x:M\to\left[-1,1\right]$ is defined by $x\left(A\right)=s$ is $n\in A$ and $x\left(A\right)=r$, otherwise. Hence, $\varphi\left(M^{*}\right)\supset K$. For the converse inclusion let $x\in M^{*}$, and $r:=x\left(\varnothing\right)$. If $x\left(\left\{m\right\}\right)=r$, for every $m\in\N$, then $\varphi\left(x\right)=y^{r,r,1}$. Otherwise, there is $n\in\N$ such that $x\left(\left\{n\right\}\right)=s>r$. Then, for every $m\ne n$ we have $x\left(\left\{m\right\}\right)\wedge s=x\left(\left\{m\right\}\right)\wedge x\left(\left\{n\right\}\right)= x\left(\varnothing\right)=r$, implying $x\left(\left\{m\right\}\right)=r$. We conclude that $\varphi\left(x\right)=y^{r,s,n}$.

Let us show that $\varphi$ is homeomorphism. If $x_{p}\to x$ in $M^{*}$, then for every $n\in\N$ we have $x_{p}\left(\left\{n\right\}\right)\to x\left(\left\{n\right\}\right)$, hence $\varphi$ is continuous. Conversely, if $\varphi\left(x_{p}\right)\to \varphi\left(x\right)$, then $x_{p}\left(\left\{n\right\}\right)\to x\left(\left\{n\right\}\right)$, for every $n\in\N$, and so $x_{p}\left(\varnothing\right)=x_{p}\left(\left\{1\right\}\right)\wedge x_{p}\left(\left\{2\right\}\right)\to x\left(\left\{1\right\}\right)\wedge x\left(\left\{2\right\}\right)= x\left(\varnothing\right)$, while for every finite $\varnothing\ne A\subset\N$ we have $x_{p}\left(A\right)=\bigvee\limits_{n\in A} x_{p}\left(\left\{n\right\}\right)\to \bigvee\limits_{n\in A} x\left(\left\{n\right\}\right)=x\left(A\right)$. We conclude that $x_{p}\to x$.\medskip

It follows that $F:=\Co_{\ph}\left(K\right)\simeq \Co_{\ph}\left(M^{*}\right)$. Let us show that $F$ is semicontinuous. Let $Y:=\left\{y^{r,s,n},~ -1\le r< s\le 1,~ n\in\N\right\}$, which is dense in $K$. Let $P\subset \Ba_{F_{+}}$ and $f\in F$ be such that $f=\bigvee P$. In order to prove that $\|f\|_{\8}\le 1$ it is enough to show that $f\left(y\right)\le 1$, for every $y\in Y$. Without loss of generality we may assume that $y=y^{r,s,1}$. Let $g,h\in F$ be defined by $g\left(x\right)=\left(x_{1}-x_{2}\right)^{+}$ and $h\left(x\right)=\left|x_{1}\right|\vee \left|x_{2}\right|$. Note that $g\left(x\right)>0$ is equivalent to the fact that $x=y^{u,v,1}$, for some $-1\le u<v\le 1$, and then $h\left(x\right)=\left|u\right|\vee \left|v\right|=\|x\|$. We claim that $e:=f-\left(f-h\right)^{+}\wedge g\ge P$. Let $x\in K$. If $g\left(x\right)=0$, then $e\left(x\right)=f\left(x\right)\ge p\left(x\right)$, for every $p\in P$. Else, we have that $h\left(x\right)=\|x\|$, and so
$$e\left(x\right)=f\left(x\right)-\left(f\left(x\right)-h\left(x\right)\right)^{+}\wedge g\left(x\right)\ge f\left(x\right)-\left(f\left(x\right)-\|x\|\right)^{+}=f\left(x\right)\wedge \|x\|\ge p\left(x\right),$$ for every $p\in P$. It follows that $e\ge f$, and so $\left(f-h\right)^{+}\wedge g=\0$. Since $g\left(y\right)=s-r>0$, we conclude that $f\left(y\right)\le h\left(y\right)=1$, as required.
\qed\end{example}

\section{Variants of $\sigma$-semicontinuity}\label{s9}

Recall that a net $\left(f_{p}\right)_{p\in P}\subset F$ \emph{$\sigma$-order} converges to $f\in F$ (denoted $f_{p}\toso f$) if there is a countable $Q\subset F$ with $Q\downarrow 0_{F}$ which dominates the tails of $\left(f_{p}-f\right)_{p\in P}$, and \emph{unbounded $\sigma$-order ($\mathrm{u\sigma o}$)} converges to $f$ (denoted $f_{p}\touso f$) if $\left|f_{p}-f\right|\wedge e\toso 0_{F}$, for every $e\in F_{+}$. Clearly, $\sigma$-order convergence implies $\mathrm{u\sigma o}$-convergence to the same limit, and the converse is true for eventually bounded nets. For an increasing sequence, its limit in both senses is its supremum (and they exist simultaneously). If a net $\sigma$-order converges, it contains a sequence which $\sigma$-order converges to the same limit. Hence, $\sigma$-order adherence coincides with the sequential $\sigma$-order adherence (the set of all existing limits of sequences in the set).

A sublattice $E\subset F$ is \emph{$\sigma$-regular} if whenever a countable $P\subset E$ is such that $\bigwedge_{E}P=0_{F}$, it then follows that $\bigwedge_{F}P=0_{F}$, and \emph{super order dense} if for every $f\in F_{+}$ there is a countable $Q\subset E_{+}$ such that $f=\bigvee Q$. The following result is similar to Lemma \ref{solid}.

\begin{lemma}\label{csolid}
Let $E\subset F$ be a super order dense sublattice and let $B\subset F$ be solid. The following sets are equal:
\item[(i)] The set of all $f\in F$ such that there is $Q\subset B_{+}$ with $Q\uparrow\left|f\right|$;
\item[(i')] The set of all $f\in F$ such that there is $Q\subset B_{+}\cap E$ with $Q\uparrow\left|f\right|$;
\item[(ii)] The order adherence of $B$ in $F$;
\item[(ii')] The order adherence of $B\cap E$ in $F$;
\item[(iii)] The uo adherence of $B$ in $F$;
\item[(iii')] The uo adherence of $B\cap E$ in $F$.\medskip

This set is solid and its intersection with $E$ is the common $\sigma$-order and $\mathrm{u\sigma o}$ adherence of $B\cap E$ in $E$. The $\sigma$-order and $\mathrm{u\sigma o}$ closures of $B$ and $B\cap E$ also coincide, and are a solid set. Both the adherence and the closure are convex, whenever $B$ is.
\end{lemma}

Note that ``countable directed sets'' can be replaced by ``increasing sequences'': if $P\subset F_{+}$ is a countable directed set, enumerate it as $\left\{p_{n}\right\}$, and then take $q_{1}:=p_{1}$, and if $q_{1},...,q_{n}\in P$ are selected, find $q_{n+1}\in P$ such that $q_{n+1}\ge q_{1},...,q_{n},p_{n+1}$.

\begin{corollary}
If $B\subset F$ is solid, then both of its $\sigma$-order and $\mathrm{u\sigma o}$ adherences consist of $f\in F$ such that there is a countable $Q\subset B_{+}$ with $Q\uparrow \left|f\right|$.
\end{corollary}

The following result is proven similarly to Corollary \ref{solid2}.

\begin{corollary}
Let $E\subset F$ be a super order dense sublattice and let $B\subset E$ be solid in $E$. Then, the $\sigma$-order and $\mathrm{u\sigma o}$ adherences of $B$ in $F$ and the $\sigma$-order and $\mathrm{u\sigma o}$ adherences of $\Sol_{F} B$ all coincide. This set is always solid, and it is convex whenever $B$ is. $f\in F$ belongs to this set if and only if there is a countable $Q\subset B_{+}$ with $Q\uparrow\left|f\right|$.
\end{corollary}

Until the end of the section let $F$ be a normed lattice. We call it \emph{$\sigma$-order continuous} if whenever $Q\subset F$ is countable and such that $Q\downarrow 0_{F}$, it then follows that $Q\to 0_{F}$ in norm, that is $\bigwedge\limits_{q\in Q}\|q\|=0$; equivalently, $f_{p}\toso 0_{F}$ implies $\|f_{p}\|\to 0$. $F$ is \emph{$\sigma$-(order) semicontinuous} if whenever countable $Q\subset\Ba_{F_{+}}$ and $f\in F$ are such that $Q\uparrow f$ it then follows that $\|f\|\le 1$, i.e. $f\in \Ba_{F}$. The following result is proven analogously to Theorem \ref{oc}.

\begin{theorem}
$F$ is $\sigma$-order continuous if and only if every renorming of $F$ is $\sigma$-semicontinuous.
\end{theorem}

The next result is similar to Proposition \ref{fatou} (Cf. \cite[Theorem 4.5]{ab0}).

\begin{proposition}
The following conditions are equivalent:
\item[(i)] $F$ is $\sigma$-semicontinuous;
\item[(ii)] $\Ba_{F}$ is $\sigma$-order closed;
\item[(ii')] $\Ba_{F}$ is $\mathrm{u\sigma o}$ closed;
\item[(iii)] If $f_{p}\toso f$, then $\|f\|\le\liminf\|f_{p}\|$;
\item[(iii')] If $f_{p}\touso f$, then $\|f\|\le\liminf\|f_{p}\|$;
\item[(iv)] If $f_{n}\toso f$, then $\|f\|\le\liminf\|f_{n}\|$;
\item[(iv')] If $f_{n}\touso f$, then $\|f\|\le\liminf\|f_{n}\|$.
\end{proposition}

The meaning of terms \emph{$\sigma$-(order) semicontinuous topology}, \emph{$r$-$\sigma$-(order) semicontinuous norm} and \emph{topologically $\sigma$-(order) semicontinuous norm} are clear. The following result is proven similarly to Proposition \ref{qf}.

\begin{proposition}
$F$ is $r$-$\sigma$-semicontinuous if and only if the common $\sigma$-order and $\mathrm{u\sigma o}$ closure of $\Ba_{F}$ is contained in $r\Ba_{F}$. Moreover, the following conditions are equivalent:
\item[(i)] $F$ is topologically $\sigma$-semicontinuous;
\item[(ii)] The norm topology is $\sigma$-semicontinuous;
\item[(iii)] The common $\sigma$-order and $\mathrm{u\sigma o}$ closure of $\Ba_{F}$ is norm bounded;
\item[(iv)] $\sigma$-order closure preserves norm bounded sets;
\item[(iv')] $\mathrm{u\sigma o}$ closure preserves norm bounded sets.
\end{proposition}

We say that $F$ is \emph{$r$-$\sigma$-(order) demicontinuous}, for some $r\ge 1$, if whenever countable $P\subset\Ba_{F_{+}}$ and $f\in F$ are such that $P\uparrow f$ it then follows that $\|f\|\le r$. Observe that $1$-$\sigma$-demicontinuity is precisely the $\sigma$-semicontinuity. We will call $F$ \emph{$\sigma$-(order) demicontinuous} if it is $r$-$\sigma$-demicontinuous, for some $r\ge 1$. This property is stable under renorming. If $F$ is $r$-$\sigma$-semicontinuous, then it is $r$-$\sigma$-demicontinuous, and so every topologically $\sigma$-semicontinuous normed lattice is $\sigma$-demicontinuous. The following two results are analogous to Propositions \ref{rwf} and \ref{wf}.

\begin{proposition}
The following conditions are equivalent:
\item[(i)] $F$ is $r$-$\sigma$-demicontinuous;
\item[(ii)] The common $\sigma$-order and $\mathrm{u\sigma o}$ adherence of $\Ba_{F}$ is contained in $r\Ba_{F}$;
\item[(iii)] Whenever $f_{p}\toso f$, then $\|f\|\le r\liminf\|f_{p}\|$;
\item[(iii')] Whenever $f_{p}\touso f$, then $\|f\|\le r\liminf\|f_{p}\|$;
\item[(iv)] Whenever $f_{n}\toso f$, then $\|f\|\le r\liminf\|f_{n}\|$;
\item[(iv')] Whenever $f_{n}\touso f$, then $\|f\|\le r\liminf\|f_{n}\|$.
\end{proposition}

\begin{proposition}
The following conditions are equivalent:
\item[(i)] $F$ is $\sigma$-demicontinuous;
\item[(ii)] The common $\sigma$-order and $\mathrm{u\sigma o}$ adherence of $\Ba_{F}$ is norm bounded;
\item[(iii)] $\sigma$-order adherence preserves norm bounded sets;
\item[(iii')] $\mathrm{u\sigma o}$ adherence preserves norm bounded sets.
\end{proposition}

As in Example \ref{dual} one shows that if $F^{*}_{\soc}:=F^{\sim}_{\soc}\cap F^{*}$ is $r$-norming, then $F$ is $r$-$\sigma$-semicontinuous.

\begin{proposition}
The $r$-$\sigma$-demicontinuity and $r$-$\sigma$-semicontinuity are inherited by $\sigma$-regular sublattices, for every $r\ge 1$. In particular, $\sigma$-semicontinuity, $\sigma$-demicontinuity and topological $\sigma$-semicontinuity pass down to $\sigma$-regular sublattices. $\sigma$-order continuity also passes down to $\sigma$-regular sublattices.
\end{proposition}

We say that $F$ is \emph{pointwise $\sigma$-(order) demicontinuous} if for every $f>0$ there is $t>0$ such that whenever $P\subset s\Ba_{F_{+}}$ is countable and such that $P\uparrow f$, it then follows that $s\ge t$. Every $\sigma$-demicontinuous normed lattice is pointwise $\sigma$-demicontinuous. Note that $F$ fails to be pointwise $\sigma$-demicontinuous if and only if there is $f>0_{F}$ such that for every $s>0$ there is a countable $P\subset s\Ba_{F_{+}}$ with $P\uparrow f$.

\begin{proposition}
\item[(i)] The class of pointwise $\sigma$-demicontinuous normed lattices is closed under $\ell_{\8}$ sums.
\item[(ii)] Let $E$ be pointwise $\sigma$-demicontinuous and let $J:F\to E$ be a continuous injective $\sigma$-order continuous homomorphism. Then, $F$ is pointwise $\sigma$-demicontinuous.
\item[(iii)] Assume that $E\subset F$ is a super order dense sublattice which is pointwise $\sigma$-demicontinuous in the induced norm. Then, $F$ is pointwise $\sigma$-demicontinuous.
\item[(iv)] $F$ is pointwise $\sigma$-demicontinuous if and only if every countable directed subset of $\Ba_{F_{+}}$ is dominable.
\end{proposition}
\begin{proof}
(i), (ii) and (iii) are proven analogously to Propositions \ref{bp}, \ref{barely2} and \ref{barely3}, respectively.

(iv): Sufficiency is proven similarly to the implication (iv)$\Rightarrow$(i) in Proposition \ref{barely}. For necessity assume that $P\subset \Ba_{F_{+}}$ is countable directed and non-dominable, so that there is $h>0_{F}$ such that $nh=\bigvee\left(P\wedge nh\right)$, for every $n\in\N$. Then, for every $n\in\N$ we have that $P_{n}:=P\wedge nh$ satisfies $P_{n}\subset\Ba_{F_{+}}$ and $P_{n}\uparrow nh$, contradicting the pointwise $\sigma$-demicontinuity.
\end{proof}

The formula $\|f\|_{\sigma L}=\bigwedge\left\{r>0,~ \exists \left(p_{n}\right)_{n\in\N}\subset r\Ba_{F_{+}} ~ p_{n}\uparrow \left|f\right|\right\}$ defines the \emph{$\sigma$-Lorentz seminorm} on $F$. By virtue of Corollary \ref{csolid2}, this seminorm is the gauge of the $\sigma$-order adherence of $\Ba_{F}$ in $F$, which is a convex solid set. Consequently, $\|\cdot\|_{\sigma L}$ is a solid seminorm with $\|\cdot\|_{L}\le \|\cdot\|_{\sigma L}\le \|\cdot\|$. It follows that $F$ is $\sigma$-semicontinuous iff $\|\cdot\|_{\sigma L}=\|\cdot\|$, $\sigma$-demicontinuous iff $\|\cdot\|_{\sigma L}$ and $\|\cdot\|$ are equivalent on $F$, and pointwise $\sigma$-demicontinuous iff $\|\cdot\|_{\sigma L}$ is a norm on $F$. This seminorm was studied in detail in \cite{cd}.

We can now introduce the concept of the anti-$\sigma$-demicontinuity. Namely $F$ has this property if $\|\cdot\|_{\sigma L}\equiv 0$. As in Lemma \ref{noatoms} one can show that if $0_{F^{*}}\ne \nu\in F^{*}_{\soc}$, then $\|f\|_{\sigma L}\ge\frac{\left|\nu\left(f\right)\right|}{\|\nu\|}$. Hence, if $F^{*}_{\sigma oc}$ separates points of $F$, then $F$ is pointwise $\sigma$-demicontinuous, and if $F$ is anti-$\sigma$-demicontinuous, then $F^{*}_{\sigma oc}=\left\{0_{F^{*}}\right\}$.

Recall that $F$ has the \emph{countable supremum property (CSP)} if whenever $P\subset F$ is such that $\bigwedge P=0_{F}$, there is a countable $P'\subset P$ such that $\bigwedge P'=0_{F}$. If $F$ has the CSP then $\|\cdot\|_{L}=\|\cdot\|_{\sigma L}$, and so every ``$\sigma$-'' concept from Sections \ref{s3} and \ref{s4} is equivalent to its ``non-$\sigma$-'' counterpart.

The normed lattices in Examples \ref{l1}, \ref{nb} and \ref{completion} are in fact anti-$\sigma$-demicontinuous. Moreover, if $\dim E=\8$, then $F:=\Co_{\ph}\left(\Ba_{E^{*}},\mathrm{w}^{*}\right)$ is also anti-$\sigma$-demicontinuous, since it is anti-demicontinuous and has the countable supremum property (see \cite{apr}).

\begin{theorem}
Let $F$ be an AM-space. Then:
\item[(i)] $F$ is $\sigma$-semicontinuous if and only if it isometrically embeds as a $\sigma$-regular sublattice of $C\left(K\right)$, for some compact Hausdorff space $K$.
\item[(ii)] $F$ is $\sigma$-demicontinuous if and only if it isomorphically embeds as a $\sigma$-regular sublattice of $C\left(K\right)$, for some compact Hausdorff space $K$. In this case it is topologically $\sigma$-semicontinuous.
\end{theorem}
\begin{proof}
(i) is proven analogously to Theorem \ref{david}. For (ii) one has to first replace the original norm on $F$ with the (equivalent) $\sigma$-Lorentz norm, which is $\sigma$-semicontinuous (follow the proof of \cite[Theorem 6]{wickstead}).
\end{proof}

\section{Countable monotone completeness and boundedness}\label{s10}

A normed lattice $F$ is \emph{countably monotonically bounded} if every countable directed subset of $\Ba_{F_{+}}$ is order bounded. \emph{$r$- and $r+$-monotone boundedness} are defined analogously. Countable monotone boundedness passes down to projection bands and majorizing sublattices.

\begin{theorem}\label{clev}
The following conditions are equivalent:
\item[(i)] $F$ is countably monotonically bounded;
\item[(ii)] $F$ is countably $r$-monotonically bounded, for some $r\ge 1$;
\item[(iii)] Every countable directed dominable subset of $\Ba_{F_{+}}$ is order bounded;
\item[(iv)] Every countable subset of $F$ which is order bounded in $F^{**}$ is also order bounded in $F$;
\item[(v)] $F$ has a countably monotonically bounded super order dense sublattice.\medskip

Moreover, the countable monotone boundedness implies the $\sigma$-demicontinuity. The countable $r+$-monotone boundedness implies the $r$-$\sigma$-demicontinuity. A super order dense countably monotonically bounded sublattice of $F$ is majorizing.
\end{theorem}
\begin{proof}
Equivalence of (i)-(iv) is proven analogously to Theorem \ref{lev} and Corollary \ref{levi}. The last two claims are proven analogously to Proposition \ref{lef}.\medskip

(i)$\Rightarrow$(v) is trivial. Let us prove the converse. Assume that $E\subset F$ is a super order dense sublattice which is countably monotonically bounded. Let $Q\subset\Ba_{F_{+}}$ be countable and directed. For every $q\in Q$ there is a countable $P_{q}\subset E_{+}$ such that $q=\bigvee P_{q}$. Let $P:=\bigcup\limits_{q\in Q}P_{q}$; then $P^{\vee}$ is a countable directed subset of $E$. To show that $P^{\vee}\subset \Ba_{F}$ use the argument from the proof of the implication (iii)$\Rightarrow$(ii) from the proof of Theorem \ref{awt}. It follows that $P^{\vee}$, and hence $P$, is order bounded in $E$ by some $e$. Then, for every $q$ we have that $P_{q}\le e$, and so $q\le e$.
\end{proof}

$F$ is called countably \emph{($r$-)monotonically complete} if every countable directed subset of $\Ba_{F_{+}}$ has a supremum in $F$ (of norm at most $r$). Every countably monotonically complete normed lattice is norm complete. Countable ($r$-)monotone completeness is equivalent to countable ($r$-)monotone boundedness together with countable order completeness. Countable $r$-monotone completeness is equivalent to countable monotone completeness together with the $r$-$\sigma$-demicontinuity. Countable ($r$-)monotone completeness and boundedness pass down to projection bands in $F$.

\begin{remark}
The reader may have noticed that the properties considered in Section \ref{s9} are related to $\sigma$-order convergence, and because of this we use the prefix ``$\sigma$-'' for naming them. In contrast to that the properties considered in Section \ref{s10} seem to be related to order convergence of sequences, and so in naming these properties we use the modifier ``countable'' instead.
\qed\end{remark}

Parts of the following result recover \cite[Theorem 2.4]{aw} and \cite[Theorem 2.5]{taylor}.

\begin{theorem}\label{clef}
The following conditions are equivalent:
\item[(i)] $F$ is countably ($r$-)monotonically complete;
\item[(ii)] Every countable directed dominable subset of $\Ba_{F_{+}}$ has a supremum in $F$ (of norm at most $r$);
\item[(iii)] $F$ is norm complete and every countable disjoint $D\subset F_{+}$ such that $D^{\vee}\subset\Ba_{F}$ has a supremum in $F$ (of norm at most $r$);
\item[(iv)] The sequential order adherence of $\Ba_{F}$ in $F^{u}$ is contained in $F$ (in $r\Ba_{F}$);
\item[(v)] Every uo-Cauchy sequence in $\Ba_{F}$ has a uo-limit (of norm at most $r$).\medskip

In particular, $F$ is countably monotonically complete and $\sigma$-semicontinuous if and only if $\Ba_{F}$ is sequentially complete with respect to uo convergence.
\end{theorem}
\begin{proof}
Equivalence of (i)-(iii), as well as the implications (iv)$\Rightarrow$(v)$\Rightarrow$(ii) are proven analogously to Theorems \ref{awt} and \ref{rawt}. Let us only sketch a proof for (ii)$\Rightarrow$(iv). First, it follows that $F$ is countably order complete. Assume that $f\in F^{u}$ is an order limit of a sequence in $\Ba_{F_{+}}$. For $n\in\N$ define $g_{n}:=\bigwedge\limits_{m\ge n}f_{m}$, which exists due to countable order completeness of $F$. It follows that $\Ba_{F_{+}}\ni g_{n}\uparrow f$, and so using (ii) we see that $f\in F$.
\end{proof}

The following result extends \cite[Lemma 2.5]{abt}.

\begin{proposition}\label{cle}
Assume that $F$ has a countable supremum property and a weak unit. Then, if $F$ is countably ($r$-)monotonically bounded or complete, then it is ($r$-)monotonically bounded or complete, respectively.
\end{proposition}
\begin{proof}
We first deal with $r$-monotone boundedness. It follows from Theorem \ref{clev} that $F$ is $\sigma$-demicontinuous, and since it has the CSP, it is in fact demicontinuous, thus pointwise demicontinuous. Let $P\subset\Ba_{F_{+}}$ be directed. By Proposition \ref{barely} it is dominable, and so it has a supremum in $F^{u}$, say $g$. Since the CSP together with a weak unit imply that $F^{u}$ has the CSP (see \cite[Theorem 6.2]{kt2}), there is a countable $P'\subset P$ such that $P'\uparrow g$. As $P'\subset\Ba_{F_{+}}$ we can apply our assumption to get $h\in F$ such that $P'\le h$ and $\|h\|\le r$. Then, $g=\bigvee P'\le h$, and so $P\le g\le h$. Thus, $F$ is $r$-monotonically bounded. The case of completeness is similar, but the last two steps are replaced with the observation that $g\in F$ and $\|g\|\le r$, because $g$ is the supremum of a countable directed subset of $\Ba_{F_{+}}$.
\end{proof}

The next example shows that the requirement of the weak unit is not superfluous in Proposition \ref{cle}.

\begin{example}
Let $X$ be an uncountable set and let $F:=\ell_{\8}^{\sigma}\left(X\right)$ be the space of all bounded functions on $X$ with a countable support. It is easy to check that $F$ is a closed ideal of $\ell_{\8}\left(X\right)$, hence an order complete semicontinuous AM-space. It is also not hard to verify its countable $1$-monotone completeness. On the other hand, it is not monotonically bounded. Indeed,  $\left\{\1_{A},~A\subset X\mbox{ -- finite}\right\}$ is a directed subset of $\Ba_{F_{+}}$, which is not order bounded.
\qed\end{example}

\begin{remark}
If $F$ is almost countably complete then $F$ has the countable universal completion, and so we can extend $\|\cdot\|_{\sigma L}$ to it. In this case $\sigma$-Lorentz completion can be defined. It will be investigated in future work.
\qed\end{remark}

\section*{Acknowledgements}

The author is grateful to Antonio Avil\'{e}s, Enrique Garc\'ia-S\'anchez, Alberto Salguero-Alarc\'on, Mitchell Taylor, Pedro Tradacete and Nazaret Trejo-Arroyo for valuable conversations on the topic of the paper, David de Hevia for conjecturing what became Theorem \ref{david}, as well as Anthony Wickstead for pointing out the necessary condition in part (i) of Theorem \ref{main}. Additionally, David Mu\~noz-Lahoz gets credit for setting up the interface that enabled interactions with models GPT-5.5 and GPT-5.6 Sol, which generated ideas for Examples \ref{secon}, \ref{nb} and \ref{completion}, Proposition \ref{urysohn}, Theorems \ref{lev} and \ref{david}, as well as a significant part of Section \ref{s8}. A part of the work on the paper was conducted during the author's visit to the University of Alberta. The author is grateful to Vladimir Troitsky for hospitality and organizing a seminar where a lot of aspects of the paper were discussed and improved. This research has been partially supported by grant PCI2024-155094-2 funded by MICIU/AEI/10.13039/501100011033

\end{document}